\documentclass[12pt,english, reqno]{amsart}

\usepackage{fullpage}
\usepackage{hyperref}
\usepackage{amssymb}
\usepackage[all]{xy}
\newcommand{\Qq}{\mathbb{Q}} 
\newcommand{\Zz}{\mathbb{Z}} 
 
\newcommand{\Cc}{\mathbb{C}} 
\newcommand{\id}{{\rm id}}

\numberwithin{equation}{section}

{\theoremstyle{plain}
\newtheorem{theo}{Th\'eor\`eme}[section]
\newtheorem{lemme}[theo]{Lemme}
\newtheorem{coro}[theo]{Corollaire}
\newtheorem{prop}[theo]{Proposition}}

{\theoremstyle{remark}
\newtheorem{rk}[theo]{Remarque}
}

\newenvironment{poliabstract}[1]
{\renewcommand{\abstractname}{#1}\begin{abstract}}
{\end{abstract}}

\title{Probl\`emes de plongement finis \\ sur les corps non commutatifs}

\author{Angelot Behajaina, Bruno Deschamps et Fran\c{c}ois Legrand}

\email{angelot.behajaina@unicaen.fr}
\address{Normandie Univ., UNICAEN, CNRS, Laboratoire de Math\'ematiques Nicolas Oresme, 14000 Caen, France}

\email{Bruno.Deschamps@univ-lemans.fr}
\address{Normandie Univ., UNICAEN, CNRS, Laboratoire de Math\'ematiques Nicolas Oresme, 14000 Caen, France et D\'epartement de Math\'ematiques, Le Mans Universit\'e, Avenue Olivier Messiaen, 72085 Le Mans cedex 9, France}

\email{francois.legrand@tu-dresden.de}
\address{Institut f\"ur Algebra, Fachrichtung Mathematik, TU Dresden, 01062 Dresden, Germany}

\begin{document}

\maketitle

\vspace{-6mm}

\begin{poliabstract}{R\'esum\'e}
Nous \'etendons la notion de probl\`eme de plongement fini sur les corps commutatifs, une notion centrale de la th\'eorie inverse de Galois, \`a la situation d'un corps quelconque $H$ de dimension finie sur son centre $h$. Nous montrons tout d'abord que r\'esoudre un probl\`eme de plongement fini sur $H$ \'equivaut \`a trouver une solution \`a un certain probl\`eme de plongement fini sur $h$ v\'erifiant une contrainte polynomiale. Nous montrons ensuite que tout probl\`eme de plongement fini scind\'e constant sur le corps de fractions rationnelles $H(t)$ \`a ind\'etermin\'ee centrale $t$ admet une solution, si $h$ est un corps ample. Il s'agit d'un analogue non commutatif d'un r\'esultat profond de Pop. Plus g\'en\'eralement, nous r\'esolvons de tels probl\`emes de plongement sur le corps  de fractions rationnelles $H(t, \sigma)$, o\`u $\sigma$ est un automorphisme de $H$ d'ordre fini. Nos r\'esultats g\'en\'eralisent de pr\'ec\'edents travaux sur le probl\`eme inverse de Galois sur les corps quelconques.
\end{poliabstract}

\section{Introduction} \label{sec:intro}

\subsection{Le probl\`eme inverse de Galois sur les corps quelconques} \label{ssec:intro_1}

La th\'eorie inverse de Galois sur un corps commutatif $k$, dont la premi\`ere question est le {\it{probl\`eme inverse de Galois}}\footnote{qui consiste \`a savoir si, pour tout groupe fini $G$, il existe un corps commutatif $\ell$ galoisien sur $k$ et tel que ${\rm{Gal}}(\ell/k)=G$.}, est un domaine de recherche bien ancr\'e en alg\`ebre et th\'eorie des nombres qui remonte \`a Hilbert et Noether. Nous renvoyons aux articles de survol \cite{DD97b, Deb01a, Deb01b} et aux ouvrages de r\'ef\'erence \cite{Ser92, Vol96, FJ08, MM18} pour un vaste panorama. Pourtant, la th\'eorie inverse de Galois peut \^etre consid\'er\'ee d\`es que l'on dispose d'une notion d'extension galoisienne. Alors qu'il s'agit d'un sujet tr\`es \'etudi\'e dans le cas commutatif, il est \'etonnant que le cas non commutatif n'ait \'et\'e, jusqu'\`a tr\`es r\'ecemment, que tr\`es peu abord\'e, alors qu'une notion naturelle d'extension galoisienne existe dans ce contexte. En effet, d'apr\`es Artin, si $H \subseteq L$ sont deux corps\footnote{Un {\it{corps}} est un anneau non nul dans lequel tout \'el\'ement non nul est inversible. Un corps dont la multiplication est commutative (resp. non commutative) est un {\it{corps commutatif}} (resp. un {\it{corps gauche}}). Dans cet article, pour mieux visualiser la distinction entre corps quelconques et corps commutatifs, nous utilisons des lettres majuscules pour les corps quelconques et des lettres minuscules pour les corps commutatifs.}, l'extension $L/H$ est {\it{galoisienne}} si le sous-corps de $L$ laiss\'e fixe par les $H$-automorphismes de $L$ vaut $H$. Nous renvoyons aux livres \cite{Jac64, Coh95} pour un panorama de la th\'eorie de Galois des corps quelconques.

Le but de cet article est de contribuer au d\'eveloppement de la th\'eorie inverse de Galois sur les corps quelconques, \`a la suite des premiers travaux \cite{DL20, ALP20, Beh21} sur le sujet. Dans \cite{DL20}, Deschamps et Legrand montrent que, si $H$ est un corps de di\-men\-sion finie sur son centre $h$, le probl\`eme inverse de Galois sur $H$ admet une r\'eponse po\-si\-tive (c'est-\`a-dire, pour tout groupe fini $G$, il existe une extension galoisienne $L/H$ de groupe $G$) si et seulement si $h$ satisfait \`a une variante du probl\`eme inverse de Galois faisant intervenir une contrainte polynomiale associ\'ee \`a la {\it{norme r\'eduite}} de $H/h$ (voir \S\ref{ssec:prelim_2} pour un \'enonc\'e plus pr\'ecis). Comme application, ils montrent que le probl\`eme inverse de Galois admet une r\'eponse positive sur le corps de fractions rationnelles $H(t)$ \`a ind\'etermin\'ee centrale $t$ (voir \S\ref{ssec:prelim_3} pour la d\'efinition), si $h$ contient un corps ample\footnote{Rappelons qu'un corps commutatif $k$ est {\it{ample}} si toute $k$-courbe lisse g\'eom\'etriquement irr\'eductible admettant au moins un point $k$-rationnel en admet en fait une infinit\'e. Les corps amples incluent les corps commutatifs alg\'ebriquement clos, les corps commutatifs valu\'es complets $\Qq_p$, $\mathbb{R}$, $\kappa((Y))$, le corps commutatif $\Qq^{\rm{tr}}$ des nombres alg\'ebriques totalement r\'eels, etc. Nous renvoyons \`a l'ouvrage de r\'ef\'erence \cite{Jar11} et aux articles de survol \cite{BSF13, Pop14} pour un vaste panorama des corps amples.}. Si $H$ est commutatif, ce dernier r\'esultat n'est rien d'autre qu'un profond r\'esultat de Pop (voir \cite[Main Theorem A]{Pop96}). Une application du deuxi\`eme r\'esultat de Deschamps et Legrand est ensuite donn\'ee dans \cite{ALP20} o\`u Alon, Legrand et Paran montrent que, si $H$ est un corps de dimension finie sur son centre $h$ et si $h$ contient un corps ample, alors le probl\`eme inverse de Galois admet une r\'eponse positive sur le corps de fractions $H(X)$ de l'anneau des fonctions polynomiales en la variable $X$. Enfin, dans \cite{Beh21}, Behajaina g\'en\'eralise le deuxi\`eme r\'esultat de Deschamps et Legrand. Sans demander que $H$ soit de dimension finie sur son centre $h$, il montre que le probl\`eme inverse de Galois admet une r\'eponse positive sur le corps de fractions rationnelles $H(t, \sigma)$, o\`u $\sigma$ est n'importe quel automorphisme de $H$ d'ordre fini tel que le sous-corps de $h$ laiss\'e fixe par $\sigma$ contienne un corps ample (si $\sigma={\rm{id}}_H$, on a $H(t, \sigma)=H(t)$).

\subsection{Probl\`emes de plongement finis sur les corps commutatifs} \label{ssec:intro_2}

Nous allons plus loin que les articles pr\'ec\'edents, en introduisant les probl\`emes de plongement finis, un sujet central dans le cas commutatif, sur n'importe quel corps de dimension finie sur son centre.

Rappelons bri\`evement ce que sont les probl\`emes de plongement finis dans le cas commutatif (voir, par exemple, \cite[\S16.4]{FJ08} pour plus de d\'etails). Un {\it{probl\`eme de plongement fini}} sur un corps commutatif $k$ est un \'epimorphisme $\alpha : G \rightarrow {\rm{Gal}}(\ell/k)$, o\`u $G$ est un groupe fini et $\ell/k$ une extension galoisienne de corps commutatifs. On dit que $\alpha$ est {\it{scind\'e}} s'il existe un plongement $\alpha' : {\rm{Gal}}(\ell/k) \rightarrow G$ tel que $\alpha \circ \alpha' = {\rm{id}}_{{\rm{Gal}}(\ell/k)}$. Une {\it{solution}} \`a $\alpha$ est un isomorphisme $\beta : {\rm{Gal}}(f/k) \rightarrow G$, o\`u $f$ est un corps commutatif contenant $\ell$ et qui est une extension galoisienne de $k$, tel que $\alpha \circ \beta$ soit l'application de restriction ${\rm{Gal}}(f/k) \rightarrow {\rm{Gal}}(\ell/k)$. Une {\it{solution g\'eom\'etrique}} \`a $\alpha$ est un isomorphisme $\beta : {\rm{Gal}}(e/k(t)) \rightarrow G$, o\`u $e$ est un corps commutatif contenant $\ell$ et qui est une extension galoisienne de $k(t)$, tel que $\alpha \circ \beta$ soit l'application de restriction ${\rm{Gal}}(e/k(t)) \rightarrow {\rm{Gal}}(\ell/k)$. 

La conjecture principale portant sur les probl\`emes de plongement finis dans le cas commutatif a \'et\'e propos\'ee par D\`ebes et Deschamps (voir \cite[\S2.2.1]{DD97b}) : {\it{tout probl\`eme de plongement fini scind\'e $G \rightarrow {\rm{Gal}}(\ell/k)$ sur n'importe quel corps commutatif $k$ poss\`ede une solution g\'eom\'etrique ${\rm{Gal}}(e/k(t)) \rightarrow G$ v\'erifiant $e \cap \overline{k}=\ell$}}. L'int\'er\^et de cette conjecture est qu'elle g\'en\'eralise et unifie plusieurs conjectures en th\'eorie inverse de Galois commutative. D'une part, elle fournit une r\'eponse positive au probl\`eme inverse de Galois sur $\Qq$. D'autre part, elle permet de r\'esoudre la {\it{conjecture de Shafarevich}}, qui affirme que le groupe de Galois absolu de l'extension cyclotomique maximale de $\Qq$ est prolibre et qui est abondamment \'etudi\'ee (voir, par exemple, \cite{Pop96, HS05, Par09, Des15}). A ce jour, la conjecture de D\`ebes et Deschamps n'a \'et\'e d\'emontr\'ee que pour les corps amples (voir \cite[Main Theorem A]{Pop96}) et aucun contre-exemple n'est connu.

\subsection{Contenu de l'article} \label{ssec:intro_3}

Concr\`etement, nous remplissons quatre objectifs.

Premi\`erement, nous \'etendons la terminologie rappel\'ee dans le \S\ref{ssec:intro_2} \`a la situation d'un corps $H$ de dimension finie sur son centre (voir \S\ref{ssec:termi_2} pour plus de d\'etails). Une premi\`ere difficult\'e est que, si $F/H$ et $L/H$ sont deux extensions galoisiennes \`a groupes de Galois finis et telles que $L \subseteq F$, alors il n'est en g\'en\'eral pas vrai que $\sigma(x) \in L$ si $\sigma \in {\rm{Gal}}(F/H)$ et $x \in L$, c'est-\`a-dire il n'y a pas, a priori, d'application de restriction ${\rm{Gal}}(F/H) \rightarrow {\rm{Gal}}(L/H)$ comme dans le cas commutatif. Cependant, une telle application de restriction existe toujours si $H$ est de dimension finie sur son centre (voir \cite[Chapter 3]{Coh95} ou \S\ref{ssec:termi_1}). De plus, pour g\'en\'eraliser la notion de solution g\'eom\'etrique, nous devons montrer au pr\'ealable que, pour une extension galoisienne $L$ d'un corps $H$ de dimension finie sur son centre \`a groupe de Galois fini, l'extension $L(t)/H(t)$ est galoisienne et il existe une application de restriction ${\rm{Gal}}(L(t)/H(t)) \rightarrow {\rm{Gal}}(L/H)$, qui est un isomorphisme (voir \S\ref{ssec:termi_1}).

Deuxi\`emement, si $\alpha : G \rightarrow {\rm{Gal}}(L/H)$ d\'esigne un probl\`eme de plongement fini sur un corps $H$ de dimension finie sur son centre $h$, nous lui associons un probl\`eme de plongement fini $\check{\alpha} : G \rightarrow {\rm{Gal}}(\ell/h)$ sur $h$, o\`u $\ell$ est le centre de $L$ (voir \eqref{def_fep_down}). Nous montrons alors le th\'eor\`eme suivant, dont le cas $L=H$ est le premier r\'esultat de \cite{DL20} rappel\'e dans le \S\ref{ssec:intro_1} :

\begin{theo} \label{thm:main_2}
Soient $\alpha : G \rightarrow {\rm{Gal}}(L/H)$ un probl\`eme de plongement fini sur un corps $H$ de dimension finie sur son centre $h$ et $\mathcal{F}_H \in h[x_1, \dots, x_n]$ la forme polynomiale associ\'ee \`a la norme r\'eduite de $H/h$ relativement au choix d'une $h$-base de $H$. Alors $\alpha$ a une solution ${\rm{Gal}}(F/H) \rightarrow G$ si et seulement si $\check{\alpha}$ a une solution ${\rm{Gal}}(f/h) \rightarrow G$ telle que $\mathcal{F}_H$ n'ait que le z\'ero trivial sur $f$.
\end{theo}

\noindent
Nous renvoyons \`a \eqref{eq:red_norm} pour la d\'efinition de la forme polynomiale $\mathcal{F}_H$ et au th\'eor\`eme \ref{thm:DL_fep} pour un \'enonc\'e plus pr\'ecis, qui fournit une correspondance explicite entre les solutions \`a $\alpha$ et celles \`a $\check{\alpha}$ v\'erifiant la contrainte polynomiale ci-dessus.

Troisi\`emement, nous \'etablissons un analogue non commutatif du r\'esultat de Pop r\'esolvant la conjecture de D\`ebes et Deschamps sur les corps amples :

\begin{theo} \label{thm:main_1}
Soit $H$ un corps de dimension finie sur son centre $h$. Tout probl\`eme de plongement fini scind\'e sur $H$ admet une solution g\'eom\'etrique, si $h$ est un corps ample.
\end{theo}

Etant donn\'es un probl\`eme de plongement fini scind\'e $\alpha : G \rightarrow {\rm{Gal}}(L/H)$ sur $H$ et un automorphisme $\sigma$ de $H$ d'ordre fini, nous donnons en fait, sous l'hypoth\`ese $h$ ample, des conditions suffisantes pour que $\alpha$ acqui\`ere une solution sur $H(t, \sigma)$. Nous renvoyons au th\'eor\`eme  \ref{thm:3} pour notre r\'esultat pr\'ecis. Ce dernier r\'esultat g\'en\'eralise, d'une part, le th\'eor\`eme \ref{thm:main_1} et permet, d'autre part, de r\'eobtenir le r\'esultat de Behajaina mentionn\'e dans le \S\ref{ssec:intro_1} dans le cas o\`u $H$ est de dimension finie sur son centre.

Quatri\`emement, rappelons que, dans le cas commutatif, la {\it{r\'eduction faible$\rightarrow$scind\'e}} (voir \cite[\S1 B) 2)]{Pop96} et \cite[\S2.1.2]{DD97b}) est un proc\'ed\'e bien connu pour d\'eduire des r\'esultats sur les probl\`emes de plongement finis admettant une {\it{solution faible}} de r\'esultats sur les probl\`emes de plongement finis scind\'es. Des applications usuelles de ce proc\'ed\'e sont que la conjecture de D\`ebes et Deschamps est \'equivalente \`a la conjecture affirmant que tout probl\`eme de plongement fini $G \rightarrow {\rm{Gal}}(\ell/k)$ sur n'importe quel corps commutatif $k$ poss\'edant une solution faible poss\`ede en fait une solution g\'eom\'etrique ${\rm{Gal}}(e/k(t)) \rightarrow G$ v\'erifiant $e \cap \overline{k}=\ell$, et que cette derni\`ere conjecture est vraie si $k$ est ample.

Nous \'etendons la notion de solution faible et la r\'eduction faible$\rightarrow$scind\'e \`a la situation des probl\`emes de plongement finis sur les corps de dimension finie sur leurs centres. Les propositions \ref{prop:fiber1} et \ref{prop:fiber2} constituent nos r\'esultats pr\'ecis. Ceux-ci nous permettent d'\'etendre le th\'eor\`eme \ref{thm:3} en montrant, en particulier, que, {\it{si $H$ est un corps de dimension finie sur son centre $h$, alors tout probl\`eme de plongement fini sur $H$ admettant une solution faible poss\`ede une solution g\'eom\'etrique, si $h$ est un corps ample}} (voir corollaire \ref{thm:4}).

\vspace{2mm}

{\bf{Remerciements.}} Nous souhaiterions remercier le rapporteur pour de pr\'ecieux commentaires qui ont permis d'am\'eliorer cet article. Ce travail de recherche s'est effectu\'e dans le cadre du projet TIGANOCO, qui est financ\'e par l'Union europ\'eenne dans le cadre du programme op\'erationnel FEDER/FSE 2014-2020.

\section{Pr\'eliminaires} \label{sec:prelim}

Dans cette partie, nous collectons le mat\'eriel sur les extensions de corps et les corps de fractions rationnelles qui sera utilis\'e dans toute la suite de l'article.

\subsection{Corps} \label{ssec:prelim_1}

Soit $H$ un corps. Pour $y \in H \setminus \{0\}$, la conjugaison dans $H$ par $y$ est not\'ee $I_H(y)$ (i.e. $I_H(y)(x)=yxy^{-1}$ pour tout $x \in H$); c'est un automorphisme {\it{int\'erieur}} de $H$. L'{\it{ordre int\'erieur}} d'un automorphisme $\sigma$ de $H$ est le plus petit $n \geq 1$ tel que $\sigma^n$ soit int\'erieur (si un tel $n$ existe), et $\infty$ sinon. Si $H$ est de dimension finie sur son centre $h$, le th\'eor\`eme de Skolem--Noether montre que l'ordre int\'erieur de $\sigma$ est l'ordre de la restriction de $\sigma$ \`a $h$.

Soit $H$ un corps de dimension finie $n^2$ sur son centre $h$. Rappelons que la {\it{norme r\'eduite}} de $H/h$ est d\'efinie comme suit. Soit $d$ un {\it{corps de d\'ecomposition}} de la $h$-alg\`ebre $H$, c'est-\`a-dire $d$ est un corps contenant $h$ tel qu'il existe un isomorphisme de $h$-alg\`ebres $\phi : H \otimes_{h} d \rightarrow \mathcal{M}_n(d)$. Alors la norme r\'eduite ${\rm{Nrd}}_{H/h}$ de $H/h$ est d\'efinie par ${\rm{Nrd}}_{H/h} ={\rm{det}} \circ \phi \circ \id : H \rightarrow d,$ o\`u ${\rm{det}} : \mathcal{M}_n(d) \rightarrow d$ et $\id : H \rightarrow H \otimes_{h} d$ sont les applications d\'eterminant et $x \mapsto x \otimes 1$. La fonction ${\rm{Nrd}}_{H/h}$ ne d\'epend, ni du choix du corps de d\'ecomposition $d$, ni de l'isomorphisme $\phi$, et on a en fait ${\rm{Nrd}}_{H/h}(x) \in h$ pour tout $x \in H$ (voir, par exemple, \cite{Bou12}). Si $e_1, \dots, e_{n^2}$ forment une $h$-base de $H$ et si $(x_1, \dots, x_{n^2}) \in h^{n^2}$, on pose
\begin{equation} \label{eq:red_norm}
\mathcal{F}_H(x_1, \dots, x_{n^2}) = {\rm{Nrd}}_{H/h}(x_1 e_1 + \cdots + x_{n^2} e_{n^2}).
\end{equation}
A partir de maintenant, on consid\`ere $\mathcal{F}_H$ comme un polyn\^ome dans $h[x_1, \dots, x_{n^2}]$.

\subsection{Extensions de corps} \label{ssec:prelim_2}

Si $L/H$ est une extension de corps (c'est-\`a-dire $H \subseteq L$), alors $L$ peut \^etre consid\'er\'e comme espace vectoriel sur $H$ \`a gauche et comme espace vectoriel sur $H$ \`a droite. Dans cet article, nous consid\'ererons $L$ comme espace vectoriel sur $H$ \`a gauche.

Soit $L/H$ une extension de corps. Le groupe des automorphismes de $L$ laissant fixe $H$ point par point est le {\it{groupe d'automorphismes}} de $L/H$, not\'e ${\rm{Aut}}(L/H)$. On dit que $L/H$ est {\it{ext\'erieure}} si le seul automorphisme int\'erieur de $L$ appartenant \`a ${\rm{Aut}}(L/H)$ est l'identit\'e ${\rm{id}}_L$ de $L$. De mani\`ere \'equivalente, si $\widetilde{H} \subseteq L$ est le {\it{commutant}} de $H$ dans $L$, c'est-\`a-dire l'ensemble des \'el\'ements $x$ de $L$ tels que $xy=yx$ pour tout $y \in H$, alors $L/H$ est ext\'erieure si et seulement si $\widetilde{H}$ est \'egal au centre de $L$. Ainsi, si $L/H$ est ext\'erieure, le centre de $H$ est contenu dans celui de $L$ et l'ordre int\'erieur d'un \'el\'ement $\sigma$ de ${\rm{Aut}}(L/H)$ vaut l'ordre de $\sigma$.

\begin{lemme} \label{lem:outer}
Soit $L/H$ une extension ext\'erieure. 

\vspace{0.5mm}

\noindent
{\rm{1)}} Si $H$ est commutatif, alors $L$ l'est aussi.

\vspace{0.5mm}

\noindent
{\rm{2)}} Pour tout corps interm\'ediaire $H \subseteq F \subseteq L$, l'extension $L/F$ est ext\'erieure.
\end{lemme}

\begin{proof}[Preuve]
1) Supposons $H$ commutatif. Comme $L/H$ est ext\'erieure, $H$ est contenu dans le centre de $L$ et tout automorphisme int\'erieur de $L$ est donc dans ${\rm{Aut}}(L/H)$. Comme $L/H$ est ext\'erieure, ${\rm{id}}_L$ est donc le seul automorphisme int\'erieur de $L$, i.e. $L$ est commutatif.

\vspace{0.5mm}

\noindent
{\rm{2)}} Si $I_L(y)$ ($y \in L^*$) fixe $F$ point par point, alors $I_L(y)$ fixe $H$ point par point. Comme $L/H$ est ext\'erieure, cela entra\^ine $I_L(y) = {\rm{id}}_L$. Ainsi $L/F$ est ext\'erieure.
\end{proof}

Comme d\'efini par Artin, une extension de corps $L/H$ est {\it{galoisienne}} si tout \'el\'ement de $L$ laiss\'e fixe par tous les \'el\'ements de ${\rm{Aut}}(L/H)$ est dans $H$. Si $L/H$ est galoisienne, le groupe d'automorphismes ${\rm{Aut}}(L/H)$ est le {\it{groupe de Galois}} de $L/H$, not\'e ${\rm{Gal}}(L/H)$.

Soient $H$ un corps et $k$ un sous-corps du centre de $H$. Soit $\ell/k$ une extension galoisienne de corps commutatifs \`a groupe de Galois fini. Supposons que $M= H \otimes_k \ell$ soit un corps. Alors, par \cite[lemme 2.1.1]{Beh21}, l'extension $M/H$ est galoisienne (on identifie $H$ et $H \otimes_k k$) et, pour tout $x \in \ell$ et tout $\sigma \in {\rm{Gal}}(M/H)$, on a $\sigma(1 \otimes x) =1 \otimes \widetilde{x}$ pour un unique $\widetilde{x} \in \ell$. De plus, $ \widetilde{\sigma} : x \mapsto \widetilde{x}$ est un \'el\'ement de ${\rm{Gal}}(\ell/k)$ et l'application suivante est en fait un isomorphisme :
\begin{equation} \label{restilde}
\widetilde{{\rm{res}}}^{M/H}_{\ell/k}: \left \{ \begin{array} {ccc}
{\rm{Gal}}(M/H) & \longrightarrow & {\rm{Gal}}(\ell/k) \\
\sigma & \longmapsto & \widetilde{\sigma}
\end{array} \right. .
\end{equation}

Le th\'eor\`eme suivant, qui repose sur \cite[th\'eor\`eme 7]{DL20} et sa preuve, d\'ecrit les extensions galoisiennes \`a groupe de Galois fini d'un corps de dimension finie sur son centre :

\begin{theo} \label{thm:DL2}
Soit $H$ un corps de dimension finie sur son centre $h$.

\vspace{0.5mm}

\noindent
{\rm{1)}} Soit $L/H$ une extension galoisienne \`a groupe de Galois fini. On a :

\vspace{0.5mm}

{\rm{a)}} le centre $\ell$ de $L$ est une extension finie galoisienne de $h$, 

\vspace{0.5mm}

{\rm{b)}} $\mathcal{F}_H$ (voir \eqref{eq:red_norm}) ne poss\`ede que le z\'ero trivial sur $\ell$,

\vspace{0.5mm}

{\rm{c)}} $L = H \otimes_{h} \ell$ \footnote{Plus pr\'ecis\'ement, l'application $x \otimes y \in H \otimes_h \ell \mapsto xy \in L$ est un isomorphisme.},

\vspace{0.5mm}

{\rm{d)}} l'application $\widetilde{{\rm{res}}}^{L/H}_{\ell/h}$ (voir \eqref{restilde}) est un isomorphisme,

\vspace{0.5mm}

{\rm{e)}} $L/H$ est ext\'erieure.

\vspace{0.5mm}

\noindent
{\rm{2)}} R\'eciproquement, soit $\ell$ un corps commutatif qui est une extension galoisienne de $h$ \`a groupe de Galois fini et sur lequel $\mathcal{F}_H$ ne poss\`ede que le z\'ero trivial. On a : 

\vspace{0.5mm}

{\rm{a)}} $H \otimes_{h} \ell$ est un corps de centre $\ell$,

\vspace{0.5mm}

{\rm{b)}} l'extension $(H \otimes_{h} \ell)/H$ est galoisienne,

\vspace{0.5mm}

{\rm{c)}} l'application $\widetilde{{\rm{res}}}^{(H \otimes_{h} \ell)/H}_{\ell/h}$ est un isomorphisme.
\end{theo}

\subsection{Corps de fractions rationnelles} \label{ssec:prelim_3}

On dit qu'un anneau int\`egre $A$ \footnote{c'est-\`a-dire $A \not= \{0\}$ et $ab \not=0$ pour tous $a \in A \setminus \{0\}$ et $b \in A \setminus \{0\}$.} est un {\it{anneau de Ore}} si, pour tous $x, y \in A \setminus \{0\}$, il existe $a, b \in A$ tels que $xa = yb \not=0$. Par \cite[Theo\-rem 6.8]{GW04}, si $A$ est un anneau de Ore, il existe un corps $H$ contenant $A$ tel que tout \'el\'ement de $H$ s'\'ecrive sous la forme $ab^{-1}$ avec $a \in A$ et $b \in A \setminus \{0\}$. De plus, par \cite[Proposition 1.3.4]{Coh95}, $H$ est unique \`a isomorphisme pr\`es et l'on dit que $H$ est le {\it{corps de fractions}} de $A$.

Soient $H$ un corps et $\sigma$ un automorphisme de $H$. L'{\it{anneau de polyn\^omes}} $H[t, \sigma]$ est l'anneau des polyn\^omes $a_0 + a_1 t + \cdots + a_n t^n$ \`a coefficients dans $H$, muni de l'addition usuelle et dont la multiplication v\'erifie $ta = \sigma(a) t$ pour tout $a \in H$. Notons que $H[t, \sigma]$ est commutatif si et seulement si $H$ l'est et $\sigma={\rm{id}}_H$. Au sens de Ore (voir \cite{Ore33}), l'anneau $H[t, \sigma]$ est l'anneau de polyn\^omes $H[t, \sigma, \delta]$ en la variable $t$, o\`u la $\sigma$-d\'erivation $\delta$ vaut 0. L'anneau $H[t, \sigma]$ est int\`egre, puisque le degr\'e du produit de deux polyn\^omes vaut la somme des degr\'es. De plus, $H[t, \sigma]$ est un anneau de Ore (voir, par exemple, \cite[Theorem 2.6 \& Corollary 6.7]{GW04}). Ainsi $H[t, \sigma]$ poss\`ede un unique corps de fractions, not\'e $H(t, \sigma)$ et appel\'e {\it{corps de fractions rationnelles}}. Si $\sigma={\rm{id}}_H$, nous \'ecrivons $H[t]$ et $H(t)$ \`a la place de $H[t, \sigma]$ et $H(t, \sigma)$. Nous consid\'erons aussi {\it{le corps $H((t, \sigma))$ des s\'eries de Laurent}} de la forme $\sum_{n \geq n_0} a_n t^n$ avec $n_0 \in \Zz$ et $a_n \in H$ pour tout $n$, muni de l'addition habituelle et dont la multiplication v\'erifie $ta = \sigma(a) t$ pour tout $a \in H$. Comme $H[t, \sigma]$ peut \^etre plong\'e dans $H((t, \sigma))$ et est un anneau de Ore, on peut plonger $H(t, \sigma)$ dans $H((t, \sigma))$. Si $\sigma={\rm{id}}_H$, nous \'ecrivons $H((t))$ au lieu de $H((t, \sigma))$. Nous renvoyons \`a \cite[\S2.3]{Coh95} pour plus de d\'etails.

\begin{lemme} \label{lemma:easy}
Soient $H$ un corps et $\sigma$ un automorphisme de $H$ d'ordre fini $m$. Notons $\widetilde{\sigma}$ la restriction de $\sigma$ au centre $h$ de $H$.

\vspace{0.5mm}

\noindent
{\rm{1)}} Si $H$ est de dimension finie sur $h$, alors $H(t, \sigma)$ est de dimension finie sur son centre.

\vspace{0.5mm}

\noindent
{\rm{2)}} Supposons que l'ordre int\'erieur de $\sigma$ soit \'egal \`a $m$. Alors le centre de $H(t,\sigma)$ vaut $h^{\langle \widetilde{\sigma} \rangle}(t^{m})$.
\end{lemme}

\begin{proof}[Preuve]
1) Comme le centre de $H(t, \sigma)$ contient $h^{\langle \widetilde{\sigma} \rangle}(t^m)$, il suffit de montrer que $H(t, \sigma)$ est de dimension finie sur $h^{\langle \widetilde{\sigma} \rangle}(t^m)$. Pour cela, notons que, comme $H$ est de dimension finie sur $h$ et $\sigma$ est d'ordre fini, la dimension de $H$ sur $h^{\langle \widetilde{\sigma} \rangle}$ est finie. Soit $e_1, \dots, e_r$ une $h^{\langle \widetilde{\sigma} \rangle}$-base de $H$ et soit $\Gamma$ l'espace vectoriel engendr\'e sur $h^{\langle \widetilde{\sigma} \rangle}(t^m)$ par tous les \'el\'ements de la forme $e_i t^j$ avec $1 \leq i \leq r$ et $0 \leq j \leq m-1$. On a alors $H[t, \sigma] \subseteq \Gamma \subseteq H(t, \sigma)$. De plus, $\Gamma$ est un anneau et est de dimension finie sur $h^{\langle \widetilde{\sigma} \rangle}(t^m)$. Ainsi $\Gamma$ est un corps. Comme tout \'el\'ement de $H(t, \sigma)$ s'\'ecrit sous la forme $P(t) Q(t)^{-1}$ avec $P(t), Q(t) \in H[t, \sigma]$, on obtient $H(t, \sigma) =\Gamma$.

\vspace{1mm}

\noindent
2) Soit $x=\sum_{n \geq n_0}a_{n}t^{n}$ un \'el\'ement du centre de $H(t,\sigma)$ avec $a_{n_0} \neq 0$. Par le {\it{crit\`ere de rationalit\'e}} (voir \cite[Proposition 2.3.3]{Coh95}), on peut trouver deux entiers $s \geq 1 $ et $n_1 \geq 0$ et des \'el\'ements $y_{1},\dots,y_{s}$ de $H$ tels que
\begin{equation} \label{eq:star}
a_n=a_{n-1}\sigma^{n-1}(y_{1})+a_{n-2}\sigma^{n-2}(y_{2})+\cdots + a_{n-s}\sigma^{n-s}(y_{s})
\end{equation}
pour tout $n \geq n_1$. Puisque $x$ est dans le centre de $H(t, \sigma)$, on a $tx=xt$, c'est-\`a-dire 
$$\sum_{n \geq n_0} \sigma(a_n)t^{n+1}=\sum_{n \geq n_0}a_nt^{n+1}.$$
Cela entra\^ine $\sigma(a_{n})=a_{n}$, c'est-\`a-dire $a_{n} \in H^{ \langle \sigma \rangle}$ pour tout $n$. On a donc $x \in H^{\langle \sigma \rangle}((t))$. De plus, pour tout $a \in H$, on a $ax=xa$, c'est-\`a-dire
$$\sum_{n \geq n_0} aa_{n}t^{n}=\sum_{n \geq n_0} a_n \sigma^{n}(a)t^{n},$$
Pour $n \geq n_0$ tel que $a_n \neq 0$, on a donc $\sigma^{n}(a)=a_{n}^{-1}a a_{n}$, i.e. $\sigma^n$ est int\'erieur. Pour un tel $n$, l'hypoth\`ese sur l'ordre int\'erieur entra\^ine alors que $m$ divise $n$. On a donc $aa_n = a_na$ pour $a \in H$ et $n$ tel que $a_n \not=0$, i.e. $x \in h^{\langle \widetilde{\sigma} \rangle }((t^{m}))$. Soit $u$ le quotient de la division euclidienne de $s$ par $m$. Par \eqref{eq:star}, pour $l \in \mathbb{N}$ tel que $ml \geq n_1$, on obtient $$a_{ml}= a_{m(l-1)}y_{m} +a_{m(l-2)}y_{2m}+ \cdots + a_{m(l-u)}y_{mu}.$$ Ainsi, par le crit\`ere de rationalit\'e, $x \in H(t^m)$. Comme $x$ est dans le centre de $H(t, \sigma)$, on obtient en fait que $x$ est dans le centre de $H(t^m)$, c'est-\`a-dire dans $h(t^m)$ (voir, par exemple, \cite[Proposition 2.1.5]{Coh95}). Ainsi $x \in h^{\langle \widetilde{\sigma} \rangle}((t^m)) \cap h(t^m)$. Comme $h^{\langle \widetilde{\sigma} \rangle}((t^m)) \cap  h = h^{\langle \widetilde{\sigma} \rangle}$ et $h/h^{\langle \widetilde{\sigma} \rangle}$ est galoisienne finie, $h^{\langle \widetilde{\sigma} \rangle}((t^m))$ et $h$ sont lin\'eairement disjoints sur $h^{\langle \widetilde{\sigma} \rangle}$. Ainsi $h(t^m)$ et $h^{\langle \widetilde{\sigma} \rangle}((t^m))$ sont lin\'eairement disjoints sur $h^{\langle \widetilde{\sigma} \rangle} (t^m)$ (voir \cite[Lemma 2.5.3]{FJ08}), i.e. $h(t^m) \cap h^{\langle \widetilde{\sigma} \rangle}((t^m)) = h^{\langle \widetilde{\sigma} \rangle} (t^m)$. On a donc $x \in h^{\langle \widetilde{\sigma} \rangle} (t^m)$.
\end{proof}

\subsection{Extensions de corps de fractions rationnelles} \label{ssec:prelim_4}

On s'int\'eresse enfin aux extensions de corps  de la forme $L(t, \tau)/H(t, \sigma)$. Commen\c cons par le lemme \'el\'ementaire suivant :

\begin{lemme} \label{triv_1}
Soient $L/H$ une extension de corps, $\sigma \in {\rm{Aut}}(H)$ et $\tau$ un automorphisme de $L$ d'ordre fini prolongeant $\sigma$.

\vspace{0.5mm}

\noindent
{\rm{1)}} Les conditions suivantes sont \'equivalentes :

\vspace{0.5mm}

{\rm{i)}} $\langle \tau, {\rm{Aut}}(L/H) \rangle = {\rm{Aut}}(L/H) \rtimes \langle \tau \rangle$,

\vspace{0.5mm}

{\rm{ii)}} ${\rm{Aut}}(L/H) \cap \langle \tau \rangle = \{{\rm{id}}_L\}$,

\vspace{0.5mm}

{\rm{iii)}} l'ordre de $\tau$ vaut l'ordre de $\sigma$.

\vspace{0.5mm}

\noindent
{\rm{2)}} Supposons les trois conditions suivantes v\'erifi\'ees :

\vspace{0.5mm}

{\rm{a)}} $L$ est commutatif,

\vspace{0.5mm}

{\rm{b)}} $L/H$ est galoisienne,

\vspace{0.5mm}

{\rm{c)}} l'ordre de $\tau$ vaut l'ordre de $\sigma$.

\vspace{0.5mm}

\noindent
Alors $L^{\langle \tau \rangle}$ et $H$ sont lin\'eairement disjoints sur $H^{\langle \sigma \rangle}$ et $L=L^{\langle \tau \rangle} H$.
\end{lemme}

\begin{proof}[Preuve]
1) Comme $\tau(H) = \sigma(H) = H$, on a ${\rm{Aut}}(L/H) \trianglelefteq \langle \tau, {\rm{Aut}}(L/H) \rangle$ et on a donc $\langle \tau, {\rm{Aut}}(L/H) \rangle = {\rm{Aut}}(L/H) \langle \tau \rangle$. Ainsi i) $\Leftrightarrow$ ii) est vraie. De plus, on a ${\rm{Aut}}(L/H) \cap \langle \tau \rangle = \langle \tau^m \rangle$, o\`u $m$ est l'ordre de $\sigma$. Par cons\'equent, ii) $\Leftrightarrow$ iii) est aussi vraie.

\vspace{0.5mm}

\noindent
2) Par a) et c), on a $[L^{\langle \tau \rangle} : H^{\langle \sigma \rangle}] = [L:H]$. De plus, $L^{\langle \tau \rangle} H = L^{\langle \tau \rangle} L^{{\rm{Gal}}(L/H)} = L^{\langle \tau \rangle \cap {\rm{Gal}}(L/H)} = L$  par b), c) et 1). Ainsi $[L^{\langle \tau \rangle} : H^{\langle \sigma \rangle}] = [L^{\langle \tau \rangle} H:H]$, ce qui ach\`eve la preuve.
\end{proof}

La proposition suivante sera utilis\'ee pour d\'efinir des applications de restriction :

\begin{prop} \label{angelot}
Soient $H$ un corps de dimension finie sur son centre $h$ et $\sigma \in {\rm{Aut}}(H)$. Notons $\widetilde{\sigma}$ la restriction de $\sigma$ \`a $h$. Soient $L/H$ une extension galoisienne \`a groupe de Galois fini et $\tau$ un automorphisme de $L$ d'ordre fini \'etendant $\sigma$. Soient $\ell$ le centre de $L$ et $\widetilde{\tau}$ la restriction de $\tau$ \`a $\ell$. Supposons la condition suivante v\'erifi\'ee :

\vspace{0.5mm}

\noindent
{\rm{($*$)}} l'ordre de $\widetilde{\tau}$ vaut l'ordre de $\widetilde{\sigma}$.

\vspace{0.5mm}

\noindent
Alors $L(t, \tau) \cong H(t, \sigma) \otimes _{h^{\langle \widetilde{\sigma} \rangle}(t^m)} \ell^{\langle \widetilde{\tau} \rangle}(t^m)$, o\`u $m$ est l'ordre de $\tau$.
\end{prop}

\begin{proof}[Preuve]
Notons tout d'abord que, par le th\'eor\`eme \ref{thm:DL2}, $\ell$ est une extension galoisienne de $h$. De plus, puisque $\tau$ \'etend $\sigma$, on a $h^{\langle \widetilde{\sigma} \rangle} \subseteq \ell^{\langle \widetilde{\tau} \rangle}$. Ainsi $H(t, \sigma) \otimes _{h^{\langle \widetilde{\sigma} \rangle}(t^m)} \ell^{\langle \widetilde{\tau} \rangle}(t^m)$ est bien d\'efini.
 
Consid\'erons maintenant l'application $h^{\langle \widetilde{\sigma} \rangle}(t^m)$-lin\'eaire
$$\psi : \left \{ \begin{array} {ccc}
H(t, \sigma) \otimes _{h^{\langle \widetilde{\sigma} \rangle}(t^m)} \ell^{\langle \widetilde{\tau} \rangle}(t^m) & \longrightarrow & L(t, \tau) \\
y \otimes z & \longmapsto & yz
\end{array}. \right. $$
Puisque $\ell^{\langle \widetilde{\tau} \rangle}(t^m)$ est contenu dans le centre de $L(t, \tau)$, $\psi$ est un morphisme d'alg\`ebres. De plus, ${\rm{Im}}(\psi)$ contient $H(t, \sigma)$ et est de dimension finie sur ce corps. Ainsi ${\rm{Im}}(\psi)$ est un corps.

Nous montrons ensuite que $\psi$ est surjective. Pour cela, notons que, par ($*$) et le lemme \ref{triv_1}, $\ell^{\langle \widetilde{\tau} \rangle}$ et $h$ sont lin\'eairement disjoints sur $h^{\langle \widetilde{\sigma} \rangle}$ et $\ell^{\langle \widetilde{\tau} \rangle} h = \ell$. Par cons\'equent, si $f_1, \dots, f_r$ est une $h^{\langle \widetilde{\sigma} \rangle}$-base de $\ell^{\langle \widetilde{\tau} \rangle}$, alors $f_1, \dots, f_r$ est une $h$-base de $\ell$. Puisque $L = H \otimes_h \ell$ (voir th\'eor\`eme \ref{thm:DL2}), on en d\'eduit que  $f_1, \dots, f_r$ (que l'on identifie \`a $1 \otimes f_1, \dots, 1 \otimes f_r$) est une $H$-base de $L$.

Maintenant, soit $x=\sum_{i}a_i t^i \in L[t,\tau]$ ($(a_i)_i \subset L$). Pour tout $i$, il existe $(a_{i,k})_{k=1}^{r} \subset H$ tel que $a_i=\sum_{k=1}^{r}a_{i,k}f_{k}$. Ainsi
$$x=\sum_{i} \left( \sum_{k=1}^{r}a_{i,k}f_{k} \right)t^i=\sum_{i} \sum_{k} \left( (a_{i,k}t^{i}) f_{k}\right) =\psi \left( \sum_{i} \sum_{k} \left( (a_{i,k}t^{i}) \otimes  f_{k}\right) \right).$$
Ainsi $L[t,\tau] \subseteq \mathrm{Im}(\psi)$. Comme $\mathrm{Im}(\psi)$ est un corps et tout \'el\'ement de $L(t, \tau)$ s'\'ecrit sous la forme $P(t) Q(t)^{-1}$ avec $P(t), Q(t) \in L[t, \tau]$, on en d\'eduit $L(t,\tau)=\mathrm{Im}(\psi)$.

Montrons enfin que $\psi$ est injective. Pour cela, notons que, puisque $H$ est de dimension finie sur $h$ et $\sigma$ est d'ordre fini, la dimension de $H$ sur $h^{\langle \widetilde{\sigma} \rangle}$ est finie. Consid\'erons une $h^{\langle \widetilde{\sigma} \rangle}$-base $(e_l)_l$ de $H$. Soit $x= \sum_{i}x_i \otimes y_i \in H(t, \sigma) \otimes _{h^{\langle \widetilde{\sigma} \rangle}(t^m)} \ell^{\langle \widetilde{\tau} \rangle}(t^m)$ ($(x_i)_i \subset H(t,\sigma)$ et $(y_i)_i \subset \ell^{\langle \widetilde{\tau} \rangle}(t^m)$) tel que $\psi(x) = 0$. Comme $\psi(x)= 0$ et $H[t, \sigma]$ est un anneau de Ore, par it\'eration de la propri\'et\'e de Ore, on peut supposer $(x_i)_i \subset H[t,\sigma]$ et $(y_i)_i \subset \ell^{\langle \widetilde{\tau} \rangle}[t^m]$. On peut alors \'ecrire
\begin{equation}\label{eq:1_ang}
x=\sum_{l} \sum_{j=0}^{m-1} \lambda_{l,j} (e_l t^j \otimes v_{l,j} ),
\end{equation} 
o\`u $(\lambda_{l,j})_{l,j} \subset h^{\langle \widetilde{\sigma} \rangle}$ et $(v_{l,j})_{l,j} \subset \ell^{\langle \widetilde{ \tau} \rangle}[t^m] \setminus \{0 \}$. Pour tout $(l,j)$, on pose $v_{l,j}=\sum_{a} b_{l,j,a}t^{ma},$ avec $(b_{l,j,a}) _{l,j,a}\subset \ell^{\langle \widetilde{\tau} \rangle}$ et, pour tout $(l,j,a)$, on a $b_{l,j,a}=\sum_{k} c_{l,j,a,k} f_{k}$ avec $(c_{l,j,a,k})_{l,j,a,k} \subset h^{\langle \widetilde{\sigma} \rangle}$. Ainsi
$$\psi(x)= \sum_{(j,a)} \left( \sum_{k} \left( \sum_{l} \left( \lambda_{l,j} c_{l,j,a,k} \right) e_{l}   \right)f_k  \right) t^{ma+j}.$$
L'\'egalit\'e pr\'ec\'edente montre alors que, pour tout $(j,a,k,l)$, on a
\begin{equation} \label{eq:2_ang}
\lambda_{l,j} c_{l,j,a,k}=0.
\end{equation}
Fixons maintenant $(l,j)$. Comme $v_{l,j} \neq 0$, il existe $(a,k)$ tel que $c_{l,j,a,k} \neq 0$. Ainsi, en utilisant \eqref{eq:2_ang}, on obtient $\lambda_{l,j}=0$. Par \eqref{eq:1_ang}, on en d\'eduit $x=0$.
\end{proof}

\begin{rk} \label{rk:bruno}
1) La condition ($*$) entra\^ine que les ordres de $\sigma$ et $\tau$ sont \'egaux. En effet, par le lemme \ref{triv_1}, il suffit de montrer ${\rm{Gal}}(L/H) \cap \langle \tau \rangle = \{{\rm{id}}_L\}$. Soit donc $j \geq 1$ tel que $\tau^j \in {\rm{Gal}}(L/H)$. En particulier, on a $\widetilde{\tau}^j \in {\rm{Gal}}(\ell/h)$. Si l'on suppose ($*$), alors, par le lemme \ref{triv_1}, on a ${\rm{Gal}}(\ell/h) \cap \langle \widetilde{\tau} \rangle = \{{\rm{id}}_\ell\}$ et donc $\widetilde{\tau}^j = {\rm{id}}_\ell$. Comme $L$ est de dimension finie sur son centre $\ell$, le th\'eor\`eme de Skolem--Noether entra\^ine que $\tau^j$ est int\'erieur. Or $\tau^j$ est dans ${\rm{Gal}}(L/H)$ et $L/H$ est ext\'erieure par le th\'eor\`eme \ref{thm:DL2}. On a donc $\tau^j = {\rm{id}}_L$.

\vspace{1mm}

\noindent
2) Si, dans la proposition \ref{angelot}, on suppose aussi que l'ordre de $\sigma$ vaut l'ordre de $\widetilde{\sigma}$, alors, par le lemme \ref{lemma:easy}, le th\'eor\`eme de Skolem--Noether et le 1), le centre de $H(t, \sigma)$ vaut $h^{\langle \widetilde{\sigma} \rangle}(t^m)$. Ainsi, dans la preuve pr\'ec\'edente, l'injectivit\'e de l'application $\psi$ est une cons\'equence directe de, par exemple, \cite[th\'eor\`eme II-3]{Bla72}.

\vspace{1mm}

\noindent
3) En g\'en\'eral, la r\'eciproque du 1) est fausse. En effet, rappelons tout d'abord que le {\it{niveau}} d'un corps commutatif $h$ est, soit le plus entier $n \geq 1$ pour lequel il existe $(x_1, \dots, x_n) \in {h^*}^n$ tel que $-1 = x_1^2 + \dots + x_n^2$ (si un tel $n$ existe), soit $\infty$ (sinon). Par le th\'eor\`eme de Pfister (voir, par exemple, \cite[Chapter XI, Theorem 2.2]{Lam05}), le niveau de $h$ est infini ou une puissance de 2. De plus, pour un corps commutatif $h$ de niveau au moins 4, notons $H_h$ le {\it{corps des quaternions}} \`a coefficients dans $h$, i.e. $H_h=h \oplus h i \oplus h j \oplus h k$ ($i^2= j^2= k^2 =i j k=-1$).

On se donne maintenant une extension galoisienne $\ell/h$ de corps commutatifs telle que $\ell$ soit de niveau au moins 4 et telle que ${\rm{Gal}}(\ell/h)$ soit fini d'ordre pair (par exemple, $h = \Qq$ et $\ell=\Qq(\sqrt{2})$). Posons $L= H_\ell$ et $H= H_h$. Par le th\'eor\`eme \ref{thm:DL2}, $L/H$ est galoisienne \`a groupe de Galois fini d'ordre pair. Soit $\tau'$ un \'el\'ement d'ordre 2 de ${\rm{Gal}}(L/H)$, soit $\sigma = I_H(i)$ et soit $\tau = I_L(i) \circ \tau'$. Clairement, $\sigma$ est d'ordre 2 et $\widetilde{\sigma} = {\rm{id}}_h$. De plus, puisque $\tau'$ fixe $H$ point par point, $\tau$ prolonge $\sigma$. Mais, puisque $L/H$ est ext\'erieure (voir th\'eor\`eme \ref{thm:DL2}), $\tau'$ n'est pas int\'erieur et, ainsi, $\tau$ ne l'est pas non plus. Par le th\'eor\`eme de Skolem--Noether, on en d\'eduit $\widetilde{\tau} \not = {\rm{id}}_\ell$. Mais, pour $x \in L$, on a $I_L(i) \circ \tau'(x)= i \tau'(x) i^{-1}$
et donc $\tau^2(x) = I_L(i) \circ \tau' (i \tau'(x) i^{-1}) = i \tau'(i) \tau'^2(x) \tau'(i^{-1}) i^{-1} = i^2 x i^{-2} = x$, ce qui montre que $\tau$ est d'ordre 2.
\end{rk}

\section{Probl\`emes de plongement finis} \label{sec:termi}

Dans cette partie, nous \'etendons la terminologie des probl\`emes de plongement finis sur les corps commutatifs \`a la situation des corps de dimension finie sur leurs centres. Nous d\'emontrons aussi une version plus pr\'ecise du th\'eor\`eme \ref{thm:main_2} (voir th\'eor\`eme \ref{thm:DL_fep}).

\subsection{Applications de restriction} \label{ssec:termi_1}

Etant donn\'ees deux extensions galoisiennes $L/H$ et $F/M$ \`a groupes de Galois finis et telles que $L \subseteq F$ et $H \subseteq M$, nous notons ${\rm{res}}^{F/M}_{L/H}$ l'application de restriction ${\rm{Gal}}(F/M) \rightarrow {\rm{Gal}}(L/H)$ (c'est-\`a-dire ${\rm{res}}^{F/M}_{L/H}(\sigma)(x)=\sigma(x)$ pour tout $\sigma \in {\rm{Gal}}(F/M)$ et tout $x \in L$), si celle-ci est bien d\'efinie.

Contrairement au cas commutatif, l'application ${\rm{res}}^{F/M}_{L/H}$ n'est pas d\'efinie en g\'en\'eral. Nous donnons maintenant des conditions suffisantes sur $L$, $H$, $F$ et $M$ pour qu'elle le soit.

\begin{prop} \label{prop:res_res_res}
Soient $L/H$ et $F/M$ deux extensions galoisiennes \`a groupes de Galois finis, telles que $L \subseteq F$ et $H \subseteq M$, et v\'erifiant les conditions suivantes :

\vspace{0.25mm}

\noindent
{\rm{1)}} $F= M \otimes_{k_M} \ell_M$ et $L = H \otimes_{k_H} \ell_H$ pour des extensions galoisiennes $\ell_M/k_M$ et $\ell_H /k _H$ de corps commutatifs \`a groupes de Galois finis et telles que $k_M$ (resp. $k_H$) soit contenu dans le centre de $M$ (resp. de $H$),

\vspace{0.25mm}

\noindent
{\rm{2)}} il existe une extension galoisienne $\ell_0/k_0$ de corps commutatifs de groupe fini telle que

\vspace{0.25mm}

{\rm{a)}} $\ell_0 \subseteq \ell_M \cap \ell_H$ et $k_0 \subseteq k_M \cap k_H$,

\vspace{0.25mm}

{\rm{b)}} l'extension $\ell_H/k_0$ est galoisienne finie et $[\ell_0:k_0]=[\ell_H : k_H]$,

\vspace{0.25mm}

{\rm{c)}} ${\rm{Gal}}(\ell_H/\ell_0) \cap {\rm{Gal}}(\ell_H/k_H)  = \{{\rm{id}}_{\ell_H}\}$.

\vspace{0.25mm}

\noindent
Alors ${\rm{res}}^{F/M}_{L/H}$ est bien d\'efinie.
\end{prop}

\begin{proof}[Preuve]
On a le diagramme d'extensions de corps suivant :
$${\small
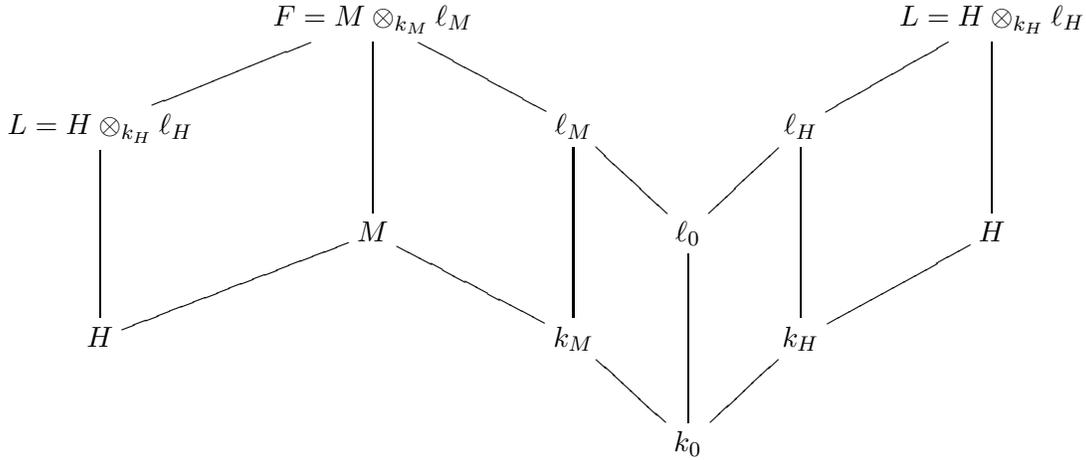
{{\xymatrix{& F =M \otimes_{k_M} \ell_M \ar@{-}[rd]  \ar@{-}[ld] \ar@{-}[dd] & & & & L = H \otimes_{k_H} \ell_H  \ar@{-}[ld] \ar@{-}[dd] \\
L = H \otimes_{k_H} \ell_H \ar@{-}[dd] & & \ell_M \ar@{-}[dd]  \ar@{-}[rd] & & \ell_H \ar@{-}[dd]  \ar@{-}[ld] & \\
&M \ar@{-}[ld]  \ar@{-}[rd]& & \ell_0 \ar@{-}[dd]  & &  \ar@{-}[ld] H\\
H && k_M \ar@{-}[rd]& & k_H  \ar@{-}[ld] & \\
& & & k_0 & & \\
}}}}$$
Puisque $\ell_0/k_0$ est galoisienne \`a groupe de Galois fini et a) est v\'erifi\'ee, on peut consid\'erer l'application ${\rm{res}}^{\ell_M/k_M}_{\ell_0/k_0}$. De plus, par b), on a $\ell_0 k_H = \ell_H^{{\rm{Gal}}(\ell_H/\ell_0)} \ell_H^{{\rm{Gal}}(\ell_H/k_H)} = \ell_H^{{\rm{Gal}}(\ell_H/\ell_0) \cap {\rm{Gal}}(\ell_H/k_H)} = \ell_H$, la derni\`ere \'egalit\'e \'etant le c). Par b), on en d\'eduit $[\ell_0 : k_0] = [\ell_0 k_H : k_H]$, c'est-\`a-dire $\ell_0$ et $k_H$ sont lin\'eairement disjoints sur $k_0$. En utilisant \`a nouveau $\ell_0 k_H = \ell_H$, on en d\'eduit que ${\rm{res}}^{\ell_H/k_H}_{\ell_0/k_0}$ est un isomorphisme. En outre, puisque 1) est vraie, les applications $\widetilde{{\rm{res}}}^{F/M}_{\ell_M/k_M}$ et $\widetilde{{\rm{res}}}^{L/H}_{\ell_H/k_H}$ (voir \eqref{restilde}) sont des isomorphismes bien d\'efinis. On peut alors consid\'erer 
\begin{equation} \label{resgeneral}
g= (\widetilde{{\rm{res}}}^{L/H}_{\ell_H/k_H})^{-1} \circ ({\rm{res}}^{\ell_H/k_H}_{\ell_0/k_0})^{-1} \circ {\rm{res}}^{\ell_M/k_M}_{\ell_0/k_0} \circ \widetilde{{\rm{res}}}^{F/M}_{\ell_M/k_M} : {\rm{Gal}}(F/M) \rightarrow {\rm{Gal}}(L/H).
\end{equation}
Alors $g={\rm{res}}^{F/M}_{L/H}$. En effet, par ce qui pr\'ec\`ede, on a $L = H \otimes_{k_H} \ell_H = H \otimes_{k_H}\ell_0 k_H$. Etant donn\'es $\sigma \in {\rm{Gal}}(F/M)$, $x \in H$, $y \in \ell_0$ et $z \in k_H$, on a donc
$$g(\sigma) (x \otimes yz) = x \otimes (({\rm{res}}^{\ell_H/k_H}_{\ell_0/k_0})^{-1} \circ {\rm{res}}^{\ell_M/k_M}_{\ell_0/k_0} \circ \widetilde{{\rm{res}}}^{F/M}_{\ell_M/k_M}(\sigma)(yz)).$$
Mais, puisque $y$ est dans $\ell_0$, on a $({\rm{res}}^{\ell_H/k_H}_{\ell_0/k_0})^{-1} \circ {\rm{res}}^{\ell_M/k_M}_{\ell_0/k_0} \circ \widetilde{{\rm{res}}}^{F/M}_{\ell_M/k_M}(\sigma)(y)=\sigma(y)$
et, puisque $z \in k_H$ et $\sigma$ laisse fixe $H$ point par point, on a
$$({\rm{res}}^{\ell_H/k_H}_{\ell_0/k_0})^{-1} \circ {\rm{res}}^{\ell_M/k_M}_{\ell_0/k_0} \circ \widetilde{{\rm{res}}}^{F/M}_{\ell_M/k_M}(\sigma)(z) = z=\sigma(z),$$
Ainsi $g(\sigma) (x \otimes yz) = x \otimes \sigma(yz) = \sigma(x \otimes yz),$ ce qui conclut la d\'emonstration.
\end{proof}

\begin{rk} \label{rk:iso}
Avec les notations de la proposition, la d\'efinition de ${\rm{res}}^{F/M}_{L/H}$ (voir \eqref{resgeneral}) montre que cette application est un isomorphisme si et seulement s'il en est de m\^eme de ${\rm{res}}^{\ell_M/k_M}_{\ell_0/k_0}$.
\end{rk}

Nous appliquons maintenant notre construction dans quatre situations.

\vspace{2mm}

\noindent
I) Bien entendu, le cas des corps commutatifs est couvert. Plus pr\'ecis\'ement, soient $F/M$ et $L/H$ deux extensions galoisiennes de corps commutatifs \`a groupes de Galois finis et telles que $H \subseteq M$ et $L \subseteq F$. Dans ce cas, on a $F= M \otimes_M F$ et $L= H \otimes_H L$. La proposition \ref{prop:res_res_res} s'applique donc avec $\ell_M=F$, $k_M=M$, $\ell_H=\ell_0=L$ et $k_H=k_0=H$.

\vspace{2mm}

\noindent
II) Soit $F/M$ une extension galoisienne \`a groupe de Galois fini. Supposons $F= M \otimes_k \ell$ pour une certaine extension galoisienne $\ell/k$ de corps commutatifs \`a groupe de Galois fini et telle que $k$ soit contenu dans le centre de $M$. On pose $\ell_M = \ell_0 = \ell_H=L = \ell$ et $k_M = k_0 = k_H=H = k$. Alors l'application ${\rm{res}}^{F/M}_{L/H}$ de la proposition \ref{prop:res_res_res} est un isomorphisme, qui vaut $\widetilde{{\rm{res}}}^{F/M}_{\ell/k}$ (voir \eqref{restilde}). En particulier, si $M$ est de dimension finie sur son centre $m$, alors $F = M \otimes_{m} f$ par le th\'eor\`eme \ref{thm:DL2}, o\`u $f$ est le centre de $F$. Ainsi ${\rm{res}}^{F/M}_{f/m}$ est un isomorphisme.

\vspace{2mm}

\noindent
III) Soient $L/H$ et $F/H$ galoisiennes \`a groupes de Galois finis avec $L \subseteq F$ et $H$ de dimension finie sur son centre $h$. Par le lemme \ref{lem:outer} et le th\'eor\`eme \ref{thm:DL2}, $F/L$ est ext\'erieure. Ainsi le centre $\ell$ de $L$ est un sous-corps du centre $f$ de $F$. De plus, par le th\'eor\`eme \ref{thm:DL2}, on a $F= H \otimes_{h}f$ (resp. $L=H \otimes_{h} \ell$) et $f/h$ (resp. $\ell/h$) est galoisienne finie. En posant $M=H$, $\ell_M=f$, $k_M=h$, $\ell_0=\ell_H=\ell$ et $k_0=k_H=h$, la proposition \ref{prop:res_res_res} montre que ${\rm{res}}^{F/H}_{L/H}$ est bien d\'efinie.

\begin{rk} \label{rk:cohn}
Si $L/H$ et $F/H$ sont deux extensions galoisiennes \`a grou\-pes de Galois finis, telles que $F/H$ soit ext\'erieure et telles que $L \subseteq F$, alors ${\rm{res}}^{F/H}_{L/H}$ est bien d\'efinie (voir \cite{Coh95}).
\end{rk}

\noindent
IV) Nous appliquons enfin la proposition \ref{prop:res_res_res} aux extensions de corps de fractions rationnelles.

\begin{coro} \label{prop:res_sigma}
Soient $H$ un corps de dimension finie sur son centre $h$ et $\sigma \in {\rm{Aut}}(H)$. Soient $L/H$ une extension galoisienne \`a groupe de Galois fini et $\tau$ un automorphisme de $L$ d'ordre fini \'etendant $\sigma$. Notons $\widetilde{\tau}$ la restriction de $\tau$ au centre $\ell$ de $L$ et supposons :
\begin{equation} \label{eq:produit}
\langle \widetilde{\tau}, {\rm{Gal}}(\ell/h) \rangle \cong \langle \widetilde{\tau} \rangle \times {\rm{Gal}}(\ell/h).
\end{equation}
Alors $L(t, \tau) / H(t, \sigma)$ est galoisienne \`a groupe de Galois fini et ${\rm{res}}^{L(t,\tau) / H(t, \sigma)}_{L/H}$ est un isomorphisme bien d\'efini.
\end{coro}

Nous aurons besoin du lemme suivant, qui sera r\'eutilis\'e dans la suite :

\begin{lemme} \label{triv_2}
Soient $H$ un corps de dimension finie sur son centre et $L/H$ une extension galoisienne \`a groupe de Galois fini. Soit $h$ (resp. $\ell$) le centre de $H$ (resp. de $L$), soit $\sigma \in {\rm{Aut}}(H)$ et soit $\tau$ un automorphisme de $L$ d'ordre fini prolongeant $\sigma$. Soit $\widetilde{\sigma}$ (resp. $\widetilde{\tau}$) la restriction de $\sigma$ (resp. de $\tau$) \`a $h$ (resp. \`a $\ell$). Les conditions suivantes sont \'equivalentes :

\vspace{0.5mm}

\noindent
{\rm{i)}} $\langle \widetilde{\tau}, {\rm{Gal}}(\ell/h) \rangle = {\rm{Gal}}(\ell/h) \times \langle \widetilde{\tau} \rangle$,

\vspace{0.5mm}

\noindent
{\rm{ii)}} l'ordre de $\widetilde{\tau}$ vaut l'ordre de $\widetilde{\sigma}$ et l'extension $\ell^{\langle \widetilde{\tau} \rangle} / h^{\langle \widetilde{\sigma} \rangle}$ est galoisienne.
\end{lemme}

\begin{proof}[Preuve du lemme \ref{triv_2}]
Par le th\'eor\`eme \ref{thm:DL2}, $\ell$ est une extension galoisienne de $h$. De plus, puisque $\tau$ prolonge $\sigma$, on a $h^{\langle \widetilde{\sigma} \rangle} \subseteq \ell^{\langle \widetilde{\tau} \rangle}$. Par cons\'equent, les conditions i) et ii) sont bien d\'efinies. De plus, en vertu du lemme \ref{triv_1}, on peut supposer que $\langle \widetilde{\tau}, {\rm{Gal}}(\ell/h) \rangle$ vaut ${\rm{Gal}}(\ell/h) \rtimes \langle \widetilde{\tau} \rangle$ et se contenter de montrer que $\langle \widetilde{\tau} \rangle \trianglelefteq \langle \widetilde{\tau}, {\rm{Gal}}(\ell/h) \rangle$ si et seulement si $\ell^{\langle \widetilde{\tau} \rangle} / h^{\langle \widetilde{\sigma} \rangle}$ est galoisienne. Pour cela, notons que, puisque $\ell/h$ est galoisienne, on a $h = \ell^{{\rm{Gal}}(\ell/h)}$ et donc $h^{\langle \widetilde{\sigma} \rangle} = \ell^{\langle \widetilde{\tau}, {\rm{Gal}}(\ell/h) \rangle}$. Comme $\langle \widetilde{\tau}, {\rm{Gal}}(\ell/h) \rangle$ est fini, le lemme d'Artin montre que $\ell / h^{\langle \widetilde{\sigma} \rangle}$ est galoisienne et ${\rm{Gal}}(\ell / h^{\langle \widetilde{\sigma} \rangle}) = \langle \widetilde{\tau}, {\rm{Gal}}(\ell/h) \rangle$. Ainsi $\ell^{\langle \widetilde{\tau} \rangle} / h^{\langle \widetilde{\sigma} \rangle}$ est galoisienne si et seulement si ${\rm{Gal}}(\ell/\ell^{\langle \widetilde{\tau} \rangle}) \trianglelefteq {\rm{Gal}}(\ell / h^{\langle \widetilde{\sigma} \rangle})$, c'est-\`a-dire si et seulement si $\langle \widetilde{\tau} \rangle \trianglelefteq \langle \widetilde{\tau}, {\rm{Gal}}(\ell/h) \rangle$.
\end{proof}

\begin{proof}[Preuve du corollaire \ref{prop:res_sigma}]
Notons $\widetilde{\sigma}$ la restriction de $\sigma$ \`a $h$. Par \eqref{eq:produit} et le lemme \ref{triv_2}, l'ordre de $\widetilde{\tau}$ vaut l'ordre de $\widetilde{\sigma}$ et $\ell^{\langle \widetilde{\tau} \rangle} / h^{\langle \widetilde{\sigma} \rangle}$ est galoisienne de degr\'e $[\ell :h]$. On peut donc appliquer la proposition \ref{angelot} et conclure que $L(t, \tau) = H(t, \sigma) \otimes _{h^{\langle \widetilde{\sigma} \rangle}(t^m)} \ell^{\langle \widetilde{\tau} \rangle}(t^m)$, o\`u $m$ est l'ordre de $\tau$. De plus, $h^{\langle \widetilde{\sigma} \rangle}(t^m)$ est contenu dans le centre de $H(t, \sigma)$ et $\ell^{\langle \widetilde{\tau} \rangle}(t^m) / h^{\langle \widetilde{\sigma} \rangle}(t^m)$ est galoisienne \`a groupe de Galois fini. Par le \S\ref{ssec:prelim_2}, $L(t, \tau)/H(t, \sigma)$ est galoisienne \`a groupe de Galois fini.

Montrons maintenant que ${\rm{res}}^{L(t,\tau) / H(t, \sigma)}_{L/H}$ est un isomorphisme bien d\'efini. On applique pour cela la proposition \ref{prop:res_res_res} avec $F=L(t, \tau)$,  $M=H(t, \sigma)$, $\ell_M = \ell^{\langle \widetilde{\tau} \rangle}(t^m)$, $k_M= h^{\langle \widetilde{\sigma} \rangle}(t^m)$, $\ell_0 = \ell^{\langle \widetilde{\tau} \rangle}$, $k_0= h^{\langle \widetilde{\sigma} \rangle}$, $\ell_H=\ell$ et $k_H=h$. L'extension $\ell_0/k_0 = \ell^{\langle \widetilde{\tau} \rangle} / h^{\langle \widetilde{\sigma} \rangle}$ est galoisienne finie et, par \eqref{eq:produit}, $\ell_H / k_0 = \ell/ h^{\langle \widetilde{\sigma} \rangle} = \ell / \ell^{\langle \widetilde{\tau}, {\rm{Gal}}(\ell/h) \rangle}$ l'est aussi. De plus, $[\ell_0 : k_0] = [\ell^{\langle \widetilde{\tau} \rangle} : h^{\langle \widetilde{\sigma} \rangle}] = [\ell :h] = [\ell_H : k_H]$ et ${\rm{Gal}}(\ell_H / \ell_0) \cap {\rm{Gal}}(\ell_H / k_H) = {\rm{Gal}}(\ell / \ell^{\langle \widetilde{\tau} \rangle}) \cap {\rm{Gal}}(\ell / h) = \{{\rm{id}}_\ell\} = \{{\rm{id}}_{\ell_H}\}$ par \eqref{eq:produit}. Ainsi ${\rm{res}}^{L(t,\tau) / H(t, \sigma)}_{L/H}$ est bien d\'efinie et, par la remarque \ref{rk:iso}, c'est un isomorphisme.
\end{proof}

\begin{rk}
Le fait que ${\rm{res}}^{L(t,\tau) / H(t, \sigma)}_{L/H}$ soit un isomorphisme bien d\'efini admet la preuve plus directe suivante, qui nous a \'et\'e en partie communiqu\'ee par le rapporteur.

Fixons tout d'abord $\rho \in {\rm{Gal}}(L(t, \tau)/H(t, \sigma))$. Alors $\rho(L) \subseteq L$. En effet, notons que $\rho$ induit un automorphisme $\widetilde{\rho}$ du centre $Z(L(t, \tau))$ de $L(t, \tau)$, qui est en fait un \'el\'ement de ${\rm{Aut}}(Z(L(t, \tau))/h^{\langle \widetilde{\sigma} \rangle}(t^m))$, o\`u $m$ est l'ordre de $\tau$. Or $\ell^{\langle \widetilde{\tau} \rangle} / h^{\langle \widetilde{\sigma} \rangle}$ est galoisienne (\`a groupe de Galois fini) par \eqref{eq:produit} et le lemme \ref{triv_2}. Par cons\'equent, $\ell^{\langle \widetilde{\tau} \rangle}(t^m) / h^{\langle \widetilde{\sigma} \rangle}(t^m)$ l'est aussi et, puisque $\ell^{\langle \widetilde{\tau} \rangle}(t^m) \subseteq Z(L(t, \tau))$, on a $\widetilde{\rho}(\ell^{\langle \widetilde{\tau} \rangle}(t^m)) = \ell^{\langle \widetilde{\tau} \rangle}(t^m)$. En particulier, pour $\alpha \in \ell^{\langle \widetilde{\tau} \rangle}$, l'\'el\'ement $\widetilde{\rho}(\alpha)$ de $\ell^{\langle \widetilde{\tau} \rangle}(t^m)$ est alg\'ebrique sur $h^{\langle \widetilde{\sigma} \rangle}$; il est donc dans $\ell^{\langle \widetilde{\tau} \rangle}$. On a donc $\rho(\ell^{\langle \widetilde{\tau} \rangle})  \subseteq \ell^{\langle \widetilde{\tau} \rangle}$. Or $\ell = \ell^{\langle \widetilde{\tau} \rangle} h$ par le lemme \ref{triv_1}, \eqref{eq:produit} et le lemme \ref{triv_2}. Par cons\'equent, on a $\rho(\ell) \subseteq \ell$. Enfin, en utilisant que l'application $x \otimes y \in H \otimes_h \ell \mapsto xy \in L$ est un isomorphisme (voir th\'eor\`eme \ref{thm:DL2}), on en d\'eduit $\rho(L) \subseteq L$, comme annonc\'e. En particulier, l'application de restriction ${\rm{res}}^{L(t,\tau) / H(t, \sigma)}_{L/H}$ est bien d\'efinie.

Ensuite, ${\rm{res}}^{L(t,\tau) / H(t, \sigma)}_{L/H}$ est injective. En effet, si $\rho \in {\rm{Gal}}(L(t, \tau)/H(t, \sigma))$ vaut l'identit\'e sur $L$, alors $\rho$ vaut l'identit\'e sur $L[t, \tau]$ (puisque $\rho(t)=t$). Comme tout \'el\'ement de $L(t, \tau)$ s'\'ecrit sous la forme $ab^{-1}$ avec $a \in L[t, \tau]$ et $b \in L[t, \tau] \setminus \{0\}$, on en d\'eduit $\rho = {\rm{id}}_{L(t, \tau)}$.

Il reste \`a montrer que ${\rm{res}}^{L(t,\tau) / H(t, \sigma)}_{L/H}$ est surjective. Notons tout d'abord que \eqref{eq:produit} entra\^ine
\begin{equation} \label{eq:angelot}
\langle \tau, {\rm{Gal}}(L/H) \rangle \cong \langle \tau \rangle \times {\rm{Gal}}(L/H).
\end{equation}
En effet, par le lemme \ref{triv_1} et \eqref{eq:produit}, les ordres de $\widetilde{\sigma}$ et $\widetilde{\tau}$ sont \'egaux. Comme vu dans le 1) de la remarque \ref{rk:bruno}, cela entra\^ine que les ordres de $\sigma$ et $\tau$ sont \'egaux. Une nouvelle utilisation du lemme \ref{triv_1} m\`ene alors \`a $\langle \tau, {\rm{Gal}}(L/H) \rangle \cong {\rm{Gal}}(L/H) \rtimes \langle \tau \rangle$. Ainsi, pour \'etablir \eqref{eq:angelot}, il ne reste plus qu'\`a montrer que $\langle \tau \rangle$ est normal dans $\langle \tau, {\rm{Gal}}(L/H) \rangle$. Fixons pour cela $\rho \in {\rm{Gal}}(L/H)$ et $x \in L$. Par le th\'eor\`eme \ref{thm:DL2}, il existe $n \geq 1$, $y_1, \dots, y_n \in H$ et $z_1, \dots, z_n \in \ell$ tels que $x = y_1 z_1 + \cdots + y_n z_n$. Ainsi 
$$\rho \tau \rho^{-1} (x) = \rho \tau \rho^{-1}(y_1) \rho \tau \rho^{-1}(z_1) + \cdots + \rho \tau \rho^{-1}(y_n) \rho \tau \rho^{-1}(z_n).$$
Fixons $i \in \{1, \dots, n\}$. Puisque $\rho$ fixe $H$ point par point, $y_i \in H$ et $\tau(H) \subseteq H$, on a $\rho \tau \rho^{-1}(y_i) = \tau(y_i)$. Ensuite, par \eqref{eq:produit}, en posant $\widetilde{\rho} = {\rm{res}}^{L/H}_{\ell/h}(\rho)$, on a $\widetilde{\rho} \widetilde{\tau} \widetilde{\rho}^{-1} = \widetilde{\tau}$ \footnote{En effet, $\widetilde{\rho} \widetilde{\tau} \widetilde{\rho}^{-1} \widetilde{\tau}^{-1}$ est \`a la fois dans $\langle \widetilde{\tau} \rangle$ (puisque $\langle \widetilde{\tau} \rangle \trianglelefteq \langle \widetilde{\tau}, {\rm{Gal}}(\ell/h) \rangle$) et dans ${\rm{Gal}}(\ell/h)$ (puisque ${\rm{Gal}}(\ell/h) \trianglelefteq \langle \widetilde{\tau}, {\rm{Gal}}(\ell/h) \rangle$. Comme $\langle \widetilde{\tau} \rangle \cap {\rm{Gal}}(\ell/h) = \{{\rm{id}}_\ell\}$, on a bien $\widetilde{\rho} \widetilde{\tau} = \widetilde{\tau} \widetilde{\rho}$.}. En particulier, puisque $z_i \in \ell$, on a $\rho \tau \rho^{-1} (z_i) = \tau(z_i)$. Ainsi $\rho \tau \rho^{-1}(x) = \tau(y_1) \tau(z_1) + \cdots + \tau(y_n) \tau(z_n) = \tau (x)$.

Fixons enfin $\rho \in {\rm{Gal}}(L/H)$. Par \eqref{eq:angelot}, $\tau$ et $\rho$ commutent et, ainsi, l'application 
$$\widetilde{\rho} : \left \{ \begin{array} {ccc}
L[t, \tau] & \longrightarrow & L[t, \tau] \\
a_0 + a_1 t+ \cdots + a_n t^n & \longmapsto & \rho(a_1) t + \cdots + \rho(a_n) t^n
\end{array} \right. $$
est un isomorphisme (voir, par exemple, \cite[Proposition 2.4]{GW04} pour plus de d\'etails), qui prolonge $\rho$ et qui fixe chaque \'el\'ement de $H[t, \sigma]$. Maintenant, pour $a, b, c, d \in L[t, \tau]$ v\'erifiant $b d \not=0$ et $ab^{-1} = c d^{-1}$, on a $\widetilde{\rho}(a) \widetilde{\rho}(b)^{-1} = \widetilde{\rho}(c) \widetilde{\rho}(d)^{-1}$. En effet, puisque $L[t, \tau]$ est un anneau de Ore, il existe $e, f$ dans $L[t, \tau] \setminus \{0\}$ tels que $b e = d f$. On a donc $ae = ab^{-1} b e = cd^{-1} df = cf$ et $\widetilde{\rho}(a) \widetilde{\rho}(b)^{-1} = \widetilde{\rho}(a) \widetilde{\rho}(e) \widetilde{\rho}(f)^{-1} \widetilde{\rho}(d)^{-1} = \widetilde{\rho}(c) \widetilde{\rho}(f) \widetilde{\rho}(f)^{-1} \widetilde{\rho}(d)^{-1} = \widetilde{\rho}(c) \widetilde{\rho}(d)^{-1},$ comme annonc\'e. Puisque tout \'el\'ement de $L(t, \tau)$ s'\'ecrit sous la forme $a b^{-1}$ avec $a \in L[t, \tau]$ et $b \in L[t, \tau] \setminus \{0\}$, l'application suivante est bien d\'efinie :
$$\widehat{\rho} : \left \{ \begin{array} {ccc}
L(t, \tau) & \longrightarrow & L(t, \tau) \\
ab^{-1} & \longmapsto & \widetilde{\rho}(a) \widetilde{\rho}(b)^{-1}
\end{array} \right.. $$
Clairement, $\widehat{\rho}$ prolonge $\widetilde{\rho}$, fixe chaque \'el\'ement de $H(t, \sigma)$ et est surjective. De plus, $\widehat{\rho}$ est un morphisme (voir, par exemple, \cite[Proposition 6.3]{GW04} pour plus de d\'etails), qui est n\'ecessairement injectif. Ainsi $\widehat{\rho} \in {\rm{Gal}}(L(t, \tau) / H(t, \sigma))$ et ${\rm{res}}^{L(t,\tau) / H(t, \sigma)}_{L/H} (\widehat{\rho}) = \rho$.
\end{rk}

\begin{rk} \label{rk:special_cases}
Pour un corps $H$ de dimension finie sur son centre, un automorphisme $\sigma$ de $H$, une extension galoisienne $L/H$ \`a groupe de Galois fini et un automorphisme $\tau$ de $L$ d'ordre fini prolongeant $\sigma$, la condition \eqref{eq:produit} est trivialement satisfaite si $\tau={\rm{id}}_L$. 

Nous exhibons maintenant trois situations moins imm\'ediates dans lesquelles la condition \eqref{eq:produit} vaut\footnote{Bien entendu, cette condition est aussi trivialement satisfaite si $L=H$ (et donc $\tau= \sigma$). Cependant, dans cette situation, la conclusion du corollaire \ref{prop:res_sigma} est sans int\'er\^et.}.

\vspace{1mm}

\noindent
1) Soit $H$ un corps de dimension finie sur son centre $h$, soit $L/H$ une extension galoisienne \`a groupe de Galois fini et soit $y \in H \setminus \{0\}$ pour lequel il existe $n \geq 1$ tel que $y^n \in h$. Clairement, le corollaire \ref{prop:res_sigma} s'applique si $\sigma = I_H(y)$ et $\tau=I_L(y)$.

\vspace{1mm}

\noindent
2) Soient $m$ et $n$ deux entiers naturels non nuls tels que $n \equiv 1 \pmod m$. Pour une ind\'etermin\'ee $u$, le corps $L= \Cc((u^{1/n}))$ est une extension galoisienne de $H=\Cc((u))$ de groupe $\Zz/n\Zz$. Si $\sigma$ est l'automorphisme de $H$ d'ordre $m$ d\'efini par $u \mapsto e^{2 i \pi /m} u$, on consid\`ere l'automorphisme $\tau$ de $L$ d\'efini par $u^{1/n} \mapsto e^{2 i \pi /m} u^{1/n}$. Alors $\tau$ est d'ordre $m$ et, puisque $n \equiv 1 \pmod m$, la restriction de $\tau$ \`a $H$ vaut $\sigma$. De plus, \eqref{eq:produit} est vraie. 

\vspace{1mm}

\noindent
3) Soit $K$ un corps de dimension finie sur son centre $k$. Soient $n \geq 2$ un entier et $G$ un groupe fini non trivial tels qu'il existe une extension galoisienne $L/K$ de groupe $\Zz/n\Zz \times G$. Par le th\'eor\`eme \ref{thm:DL2}, si $\ell$ est le centre de $L$, on a :

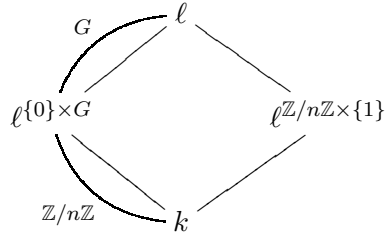
$${\small{{\xymatrix{ & \ell\ar@{-}[rd] \ar@{-} @/_1pc/[ld]_{G} \ar@{-}[ld] & \\
\ell^{\{0\} \times G} \ar@{-}[rd] \ar@{-} @/_1pc/[rd]_{\Zz/n\Zz} & & \ell^{\Zz/n\Zz \times \{1\}} \ar@{-}[ld] \\
& k & \\
}}}}$$
De plus, $H= K \otimes_{k} \ell^{\{0\} \times G}$ est un corps et $H/K$ est galoisienne de groupe $\Zz/n\Zz$. Soient $\widetilde{\sigma}$ un g\'en\'erateur de ${\rm{Gal}}(\ell^{\{0\} \times G}/k)$ et $\sigma$ un g\'en\'erateur de ${\rm{Gal}}(H/K)$ tel que ${\rm{res}}^{H/K}_{\ell^{\{0\} \times G}/k}(\sigma)=\widetilde{\sigma}$. En outre, $\ell^{\{0\} \times G}$ et $\ell^{\Zz/n\Zz \times \{1\}}$ sont lin\'eairement disjoints sur $k$ et leur compositum vaut $\ell$. Ainsi il existe $\widetilde{\tau} \in {\rm{Gal}}(\ell/k)$ d'ordre $n$ et tel que ${\rm{res}}^{\ell/k}_{\ell^{\{0\} \times G}/k}(\widetilde{\tau})=\widetilde{\sigma}$. Soit $\tau \in {\rm{Gal}}(L/K)$ l'ant\'ec\'edent de $\widetilde{\tau}$ sous ${\rm{res}}^{L/K}_{\ell/k}$. Alors $\tau$ \'etend $\sigma$. Enfin, $L/H$ est galoisienne de groupe $G$ (voir th\'eor\`eme \ref{thm:DL2}) et, par construction, on a $\langle \widetilde{\tau}, {\rm{Gal}}(\ell/\ell^{\{0\} \times G}) \rangle \cong \langle \widetilde{\tau} \rangle \times {\rm{Gal}}(\ell/\ell^{\{0\} \times G}).$
\end{rk}

Le corollaire \ref{prop:res_sigma} admet la r\'eciproque partielle suivante :

\begin{prop} \label{prop:converse}
Soient $H$ un corps de dimension finie sur son centre $h$ et $\sigma \in {\rm{Aut}}(H)$. Soient $L/H$ une extension galoisienne \`a groupe de Galois fini et $\tau$ un automorphisme de $L$ d'ordre fini \'etendant $\sigma$. Notons $\widetilde{\tau}$ la restriction de $\tau$ au centre $\ell$ de $L$. Supposons :

\vspace{0.5mm}

\noindent
{\rm{1)}} $L(t, \tau)/H(t, \sigma)$ est galoisienne \`a groupe de Galois fini,

\vspace{0.5mm}

\noindent
{\rm{2)}} l'ordre int\'erieur de $\sigma$ (resp. de $\tau$) vaut l'ordre de $\sigma$ (resp. de $\tau$).

\vspace{0.5mm}

\noindent
Alors $\langle \widetilde{\tau}, {\rm{Gal}}(\ell/h) \rangle \cong \langle \widetilde{\tau} \rangle \times {\rm{Gal}}(\ell/h).$
\end{prop}

\begin{proof}[Preuve]
Soit $\widetilde{\sigma}$ la restriction de $\sigma$ \`a $h$. Par 2) et le lemme \ref{lemma:easy}, les centres de $H(t, \sigma)$ et $L(t, \tau)$ sont $h^{\langle \widetilde{\sigma} \rangle}(t^m)$ et $\ell^{\langle \widetilde{\tau} \rangle }(t^n)$, o\`u $m$ (resp. $n$) est l'ordre de $\sigma$ (resp. de $\tau$). Comme $H(t, \sigma)$ est de dimension finie sur son centre (voir lemme \ref{lemma:easy}), le th\'eor\`eme \ref{thm:DL2} et 1) montrent que $\ell^{\langle \widetilde{\tau} \rangle }(t^n)/h^{\langle \widetilde{\sigma} \rangle }(t^m)$ est galoisienne. De plus, comme $t^m \in \ell^{\langle \widetilde{\tau} \rangle}(t^n)$, on a $n \, | \,  m$ et donc $n=m$. Ainsi $\ell^{\langle \widetilde{\tau} \rangle}/h^{\langle \widetilde{\sigma} \rangle}$ est galoisienne. En outre, comme $n=m$, le th\'eor\`eme de Skolem--Noether et 2) montrent que les ordres de $\widetilde{\sigma}$ et $\widetilde{\tau}$ sont \'egaux. Le lemme \ref{triv_2} permet alors de conclure.
\end{proof}

\subsection{Probl\`emes de plongement finis} \label{ssec:termi_2}

\subsubsection{Terminologie} \label{sssec:fep_1}

Soit $H$ un corps de dimension finie sur son centre. Un {\it{probl\`eme de plongement fini}} sur $H$ est un \'epimorphisme $\alpha : G \rightarrow {\rm{Gal}}(L/H)$, o\`u $G$ est un groupe fini et $L/H$ une extension galoisienne. On dit que $\alpha$ est {\it{scind\'e}} s'il existe un plongement $\alpha' : {\rm{Gal}}(L/H) \rightarrow G$ tel que $\alpha \circ \alpha' = {\rm{id}}_{{\rm{Gal}}(L/H)}$. Une {\it{solution faible}} \`a $\alpha$ est un monomorphisme $\beta : {\rm{Gal}}(F/H) \rightarrow G$, o\`u $F/H$ est une extension galoisienne v\'erifiant $L \subseteq F$, tel que l'application compos\'ee $\alpha \circ \beta$ soit l'application de restriction ${\rm{res}}^{F/H}_{L/H}$ du III) du \S\ref{ssec:termi_1}. Si $\beta$ est un isomorphisme, on dit {\it{solution}} plut\^ot que solution faible.

\begin{rk} \label{rk:fep}
1) Plus g\'en\'eralement, on peut d\'efinir un probl\`eme de plongement fini sur un corps quelconque $H$ (pas n\'ecessairement de dimension finie sur son centre). Cependant, pour la notion de solution faible, nous avons besoin que ${\rm{res}}^{F/H}_{L/H}$ soit bien d\'efinie. Comme vu dans la remarque \ref{rk:cohn}, c'est le cas sans aucune hypoth\`ese sur $H$, \`a condition de supposer $F/H$ ext\'erieure. Comme la plupart de nos r\'esultats supposent $H$ de dimension finie sur son centre, nous nous restreignons \`a cette situation dans la terminologie ci-dessus.

\vspace{1mm}

\noindent
2) Soient $H$ un corps de dimension finie sur son centre et $\alpha : G \rightarrow {\rm{Gal}}(L/H)$ un probl\`eme de plongement fini sur $H$. Puisque $G$ est fini, $L/H$ est ext\'erieure (voir th\'eor\`eme \ref{thm:DL2}) et, ainsi, le lemme \ref{lem:outer} s'applique : si $H$ est commutatif, alors $L$ l'est aussi. De m\^eme, si ${\rm{Gal}}(F/H) \rightarrow G$ est une solution faible \`a $\alpha$ et si $H$ est commutatif, alors $F$ l'est aussi. Notre terminologie \'etend donc celle du cas commutatif rappel\'ee dans le \S\ref{ssec:intro_2} et aucune confusion n'est possible.
\end{rk}

\subsubsection{Passage des corps quelconques aux corps commutatifs et vice-versa} \label{sssec:fep_2}

Soit $H$ un corps de dimension finie sur son centre $h$. Nous renvoyons \`a \eqref{eq:red_norm} pour la d\'efinition de la forme polynomiale $\mathcal{F}_H$ associ\'ee \`a la norme r\'eduite de $H/h$. 

Soit $\alpha : G \rightarrow {\rm{Gal}}(L/H)$ un probl\`eme de plongement fini sur $H$. Puisque ${\rm{res}}_{\ell/h}^{L/H}$ est un isomorphisme, o\`u $\ell$ est le centre de $L$ (voir le II) du \S\ref{ssec:termi_1}),
\begin{equation} \label{def_fep_down}
\check{\alpha} = {\rm{res}}_{\ell/h}^{L/H} \circ \alpha : G \rightarrow {\rm{Gal}}(\ell/h)
\end{equation}
est un probl\`eme de plongement fini sur $h$. Par le th\'eor\`eme \ref{thm:DL2}, $\mathcal{F}_H$ n'a que le z\'ero trivial sur $\ell$. De plus, si $\beta : {\rm{Gal}}(F/H) \rightarrow G$ est une solution faible \`a $\alpha$, alors ${\rm{res}}^{F/H}_{f/h}$ est un isomorphisme, o\`u $f$ est le centre de $F$. Ainsi 
\begin{equation} \label{def_sol_down}
\check{\beta}=\beta \circ ({\rm{res}}^{F/H}_{f/h})^{-1} : {\rm{Gal}}(f/h) \rightarrow G
\end{equation}
est une solution faible \`a $\check{\alpha}$, car ${\rm{res}}^{L/H}_{\ell/h} \circ {\rm{res}}^{F/H}_{L/H} = {\rm{res}}^{f/h}_{\ell/h} \circ {\rm{res}}^{F/H}_{f/h}$. De plus, $\mathcal{F}_H$ n'a que le z\'ero trivial sur $f$ et $\beta$ est une solution \`a $\alpha$ si et seulement si $\check{\beta}$ est une solution \`a $\check{\alpha}$.

R\'eciproquement, soit $\alpha : G \rightarrow {\rm{Gal}}(\ell/h)$ un probl\`eme de plongement fini sur $h$ tel que $\mathcal{F}_H$ ne poss\`ede que le z\'ero trivial sur $\ell$. Par le th\'eor\`eme \ref{thm:DL2}, $(H \otimes_{h} \ell)/ H$ est galoisienne \`a groupe de Galois fini et le centre de $H \otimes_{h} \ell$ vaut $\ell$. Ainsi ${\rm{res}}^{(H \otimes_{h} \ell)/H}_{\ell/h}$ est un isomorphisme et
\begin{equation} \label{def_fep_up}
\hat{\alpha} = ({\rm{res}}^{(H \otimes_{h} \ell)/H}_{\ell/h})^{-1} \circ \alpha : G \rightarrow {\rm{Gal}}((H \otimes_{h} \ell)/H)
\end{equation}
est un probl\`eme de plongement fini sur $H$. De plus, si $\beta : {\rm{Gal}}(f/h) \rightarrow G$ est une solution faible \`a $\alpha$ telle que $\mathcal{F}_H$ ne poss\`ede que le z\'ero trivial sur $f$, alors, par le th\'eor\`eme \ref{thm:DL2}, $(H \otimes_{h} f)/ H$ est une extension galoisienne \`a groupe de Galois fini et le centre de $H \otimes_{h} f$ vaut $f$. Ainsi ${\rm{res}}^{(H \otimes_{h} f)/H}_{f/h}$ est un isomorphisme. Par cons\'equent, l'application compos\'ee
\begin{equation} \label{def_sol_up}
\hat{\beta}=\beta \circ {\rm{res}}^{(H \otimes_{h} f)/H}_{f/h} : {\rm{Gal}}((H \otimes_{h} f)/H) \rightarrow G
\end{equation}
est une solution faible \`a $\hat{\alpha}$, puisque ${\rm{res}}^{(H \otimes_{h} \ell)/H}_{\ell/h} \circ {\rm{res}}^{(H \otimes_{h} f)/H}_{(H \otimes_{f} \ell)/H} = {\rm{res}}^{f/h}_{\ell/h} \circ {\rm{res}}_{f/h}^{(H \otimes_{h} f)/H}$. Il est \'egalement clair que $\beta$ est une solution \`a $\alpha$ si et seulement si $\hat{\beta}$ est une solution \`a $\hat{\alpha}$.

Enfin, soit $\alpha : G \rightarrow {\rm{Gal}}(L/H)$ un probl\`eme de plongement fini sur $H$. Puisque $L = H \otimes_{h} \ell$ (voir th\'eor\`eme \ref{thm:DL2}), o\`u $h$ (resp. $\ell$) est le centre de $H$ (resp. de $L$), on a $\hat{\check{\alpha}}=\alpha$. De m\^eme, si $\beta : {\rm{Gal}}(F/H) \rightarrow G$ est une solution faible \`a $\alpha$, alors $\hat{\check{\beta}}=\beta$. R\'eciproquement, soit $\alpha : G \rightarrow {\rm{Gal}}(\ell/h)$ un probl\`eme de plongement fini sur $h$ tel que $\mathcal{F}_H$ ne poss\`ede que le z\'ero trivial sur $\ell$. Puisque le centre de $H \otimes_{h} \ell$ vaut $\ell$, on a $\check{\hat{\alpha}}=\alpha$. De m\^eme, si $\beta : {\rm{Gal}}(f/h) \rightarrow G$ est une solution faible \`a $\alpha$ telle que $\mathcal{F}_H$ ne poss\`ede que le z\'ero trivial sur $f$, alors $\check{\hat{\beta}}=\beta$.

Nous avons donc d\'emontr\'e le th\'eor\`eme suivant, qui g\'en\'eralise \`a la fois \cite[th\'eor\`eme 7]{DL20} et le th\'eor\`eme  \ref{thm:main_2} de l'introduction :

\begin{theo} \label{thm:DL_fep}
Soient $H$ un corps de dimension finie sur son centre $h$ et $\alpha : G \rightarrow {\rm{Gal}}(L/H)$ un probl\`eme de plongement fini sur $H$. Les applications $\beta \mapsto \check{\beta}$ et $\gamma \mapsto \hat{\gamma}$ sont des bijections r\'eciproques entre l'ensemble des solutions $\beta$ \`a $\alpha$ et l'ensemble des solutions $\gamma : {\rm{Gal}}(f/h) \rightarrow G$ \`a $\check{\alpha}$ telles que $\mathcal{F}_H$ ne poss\`ede que le z\'ero trivial sur $f$.
\end{theo}

\subsubsection{Solutions g\'eom\'etriques} \label{sssec:fep_3}

Soient $H$ un corps de dimension finie sur son centre $h$ et $\sigma \in {\rm{Aut}}(H)$. Soient $\alpha : G \rightarrow {\rm{Gal}}(L/H)$ un probl\`eme de plongement fini sur $H$ et $\tau$ un automorphisme de $L$ d'ordre fini \'etendant $\sigma$. Notons $\widetilde{\tau}$ la restriction de $\tau$ au centre $\ell$ de $L$ et supposons $\langle \widetilde{\tau}, {\rm{Gal}}(\ell/h) \rangle \cong \langle \widetilde{\tau} \rangle \times {\rm{Gal}}(\ell/h)$. Alors $L(t, \tau) / H(t, \sigma)$ est galoisienne \`a groupe de Galois fini et ${\rm{res}}^{L(t,\tau) / H(t, \sigma)}_{L/H}$ est un isomorphisme bien d\'efini (voir corollaire \ref{prop:res_sigma}). Comme le corps $H(t, \sigma)$ est de dimension finie sur son centre (voir lemme \ref{lemma:easy}), l'application compos\'ee
\begin{equation} \label{eq:fep_geo}
\alpha_{\sigma, \tau} = ({\rm{res}}^{L(t, \tau)/H(t, \sigma)}_{L/H})^{-1} \circ \alpha : G \rightarrow {\rm{Gal}}(L(t, \tau)/H(t, \sigma))
\end{equation}
est un probl\`eme de plongement fini sur $H(t, \sigma)$. Une {\it{solution $(\sigma, \tau)$-g\'eom\'etrique}} \`a $\alpha$ est une solution ${\rm{Gal}}(E/H(t, \sigma)) \rightarrow G$ \`a $\alpha_{\sigma, \tau}$. Si $\tau={\rm{id}}_L$, on dit plut\^ot {\it{solution g\'eom\'etrique}}. Comme dans la remarque \ref{rk:fep}, $E$ est commutatif d\`es que $H(t, \sigma)$ l'est. En particulier, il n'y a aucune confusion possible avec la notion de solution g\'eom\'etrique du \S\ref{ssec:intro_2}.

De plus, notons $\widetilde{\sigma}$ la restriction de $\sigma$ \`a $h$. Par les lemmes \ref{triv_1} et \ref{triv_2}, l'extension $\ell^{\langle \widetilde{\tau} \rangle} / h^{\langle \widetilde{\sigma} \rangle}$ est galoisienne et ${\rm{res}}^{\ell/h}_{\ell^{\langle \widetilde{\tau} \rangle} / h^{\langle \widetilde{\sigma} \rangle}}$ est un isomorphisme. Avec $\check{\alpha}$ comme dans \eqref{def_fep_down}, on peut donc consid\'erer le probl\`eme de plongement fini sur $h^{\langle \widetilde{\sigma} \rangle}$ suivant :
\begin{equation} \label{eq:fep_overline}
\overline{\alpha}_{\sigma, \tau} = {\rm{res}}^{\ell/h}_{\ell^{\langle \widetilde{\tau} \rangle} / h^{\langle \widetilde{\sigma} \rangle}} \circ \check{\alpha} : G \rightarrow {\rm{Gal}}(\ell^{\langle \widetilde{\tau} \rangle} / h^{\langle \widetilde{\sigma} \rangle}).
\end{equation}

\begin{lemme} \label{lemma:link}
On a $\alpha_{\sigma, \tau} = ({\rm{res}}^{L(t, \tau)/H(t, \sigma)}_{\ell^{\langle \widetilde{\tau} \rangle}/h^{\langle \widetilde{\sigma} \rangle}})^{-1} \circ \overline{\alpha}_{\sigma, \tau}.$
\end{lemme}

\begin{proof}[Preuve]
Notons $m$ l'ordre de $\tau$. Par la d\'efinition de ${\rm{res}}^{L(t, \tau)/H(t, \sigma)}_{L/H}$, on a
$$\begin{array} {ccl}
\alpha_{\sigma, \tau} & = & ({\rm{res}}^{L(t, \tau)/H(t, \sigma)}_{L/H})^{-1} \circ \alpha  \\
& = & ({\rm{res}}^{L(t, \tau)/H(t, \sigma)}_{\ell^{\langle \widetilde{\tau} \rangle}(t^m)/h^{\langle \widetilde{\sigma} \rangle}(t^m)})^{-1} \circ ({\rm{res}}^{\ell^{\langle \widetilde{\tau} \rangle}(t^m)/h^{\langle \widetilde{\sigma} \rangle}(t^m)}_{\ell^{\langle \widetilde{\tau} \rangle}/h^{\langle \widetilde{\sigma} \rangle}})^{-1} \circ {\rm{res}}^{\ell/h}_{\ell^{\langle \widetilde{\tau} \rangle}/h^{\langle \widetilde{\sigma} \rangle}} \circ \check{\alpha} \\
&= & ({\rm{res}}^{L(t, \tau)/H(t, \sigma)}_{\ell^{\langle \widetilde{\tau} \rangle}(t^m)/h^{\langle \widetilde{\sigma} \rangle}(t^m)})^{-1} \circ ({\rm{res}}^{\ell^{\langle \widetilde{\tau} \rangle}(t^m)/h^{\langle \widetilde{\sigma} \rangle}(t^m)}_{\ell^{\langle \widetilde{\tau} \rangle}/h^{\langle \widetilde{\sigma} \rangle}})^{-1} \circ  \overline{\alpha}_{\sigma, \tau}\\
& = & ({\rm{res}}^{L(t, \tau)/H(t, \sigma)}_{\ell^{\langle \widetilde{\tau} \rangle}/h^{\langle \widetilde{\sigma} \rangle}})^{-1} \circ \overline{\alpha}_{\sigma, \tau},
\end{array}$$
ce qui d\'emontre le lemme.
\end{proof}

\section{R\'esolution de probl\`emes de plongement finis scind\'es} \label{sec:main}

Le th\'eor\`eme suivant est l'objectif de cette partie :

\begin{theo} \label{thm:3}
Soit $H$ un corps de dimension finie sur son centre $h$, soit $\sigma \in {\rm{Aut}}(H)$ et soit $\widetilde{\sigma}$ la restriction de $\sigma$ \`a $h$. Soient $\alpha : G \rightarrow {\rm{Gal}}(L/H)$ un probl\`eme de plongement fini sur $H$ et $\tau$ un automorphisme de $L$ d'ordre fini \'etendant $\sigma$. Soit $\widetilde{\tau}$ la restriction de $\tau$ au centre $\ell$ de $L$. Supposons que les trois conditions suivantes soient v\'erifi\'ees :

\vspace{0.5mm}

\noindent
{\rm{1)}} $\alpha$ est scind\'e,

\vspace{0.5mm}

\noindent
{\rm{2)}} $\langle \widetilde{\tau}, {\rm{Gal}}(\ell/h) \rangle \cong \langle \widetilde{\tau} \rangle \times {\rm{Gal}}(\ell/h)$,

\vspace{0.5mm}

\noindent
{\rm{3)}} il existe un sous-corps ample $k_0$ de $h^{\langle \widetilde{\sigma} \rangle}$ et un corps commutatif $\ell_0$ galoisien sur $k_0$ tels que $\ell_0$ et $h^{\langle \widetilde{\sigma} \rangle}$ soient lin\'eairement disjoints sur $k_0$ et tels que $\ell_0 h^{\langle \widetilde{\sigma} \rangle} = \ell^{\langle \widetilde{\tau} \rangle}$.

\vspace{0.5mm}

\noindent
Alors $\alpha$ admet une solution $(\sigma, \tau)$-g\'eom\'etrique.
\end{theo}

Nous aurons besoin du lemme suivant, qui fait le lien avec le cas des corps commutatifs :

\begin{lemme} \label{lemma:DL}
En conservant les notations du th\'eor\`eme, supposons que {\rm{2)}} soit v\'erifi\'ee et notons $m$ l'ordre de $\tau$. Alors $\alpha$ a une solution $(\sigma, \tau)$-g\'eom\'etrique, pourvu que $\overline{\alpha}_{\sigma, \tau}$ (voir \eqref{eq:fep_overline}) poss\`ede une solution g\'eom\'etrique ${\rm{Gal}}(e/h^{\langle \widetilde{\sigma} \rangle}(t^m)) \rightarrow G$ v\'erifiant $e \subseteq \ell^{\langle \widetilde{\tau} \rangle} ((t^m))$.
\end{lemme}

\begin{proof}[Preuve du lemme \ref{lemma:DL}]
Tout d'abord, notons que $\overline{\alpha}_{\sigma, \tau}$ est bien d\'efini puisque 2) est vraie. Soit $\beta : {\rm{Gal}}(e/h^{\langle \widetilde{\sigma} \rangle}(t^m)) \rightarrow G$ une solution g\'eom\'etrique \`a $\overline{\alpha}_{\sigma, \tau}$ telle que $e \subseteq \ell^{\langle \widetilde{\tau} \rangle} ((t^m))$. Puisque la derni\`ere inclusion vaut, \cite[proposition 2.1.2]{Beh21} s'applique : $L(t, \tau) \otimes_{\ell^{\langle \widetilde{\tau} \rangle}(t^m)} e$ est un corps. De plus, puisque 2) est vraie, on a
\begin{equation} \label{eq:tensor_product}
L(t, \tau) = H(t, \sigma) \otimes_{h^{\langle \widetilde{\sigma} \rangle}(t^m)} \ell^{\langle \widetilde{\tau} \rangle}(t^m)
\end{equation}
(voir proposition \ref{angelot}). Ainsi, en vertu de, par exemple, \cite[proposition 37]{Des12}, on a
$$L(t, \tau) \otimes_{\ell^{\langle \widetilde{\tau} \rangle}(t^m)} e =(H(t, \sigma) \otimes_{h^{\langle \widetilde{\sigma} \rangle}(t^m)} \ell^{\langle \widetilde{\tau} \rangle}(t^m)) \otimes_{\ell^{\langle \widetilde{\tau} \rangle}(t^m)} e = H(t, \sigma)  \otimes_{h^{\langle \widetilde{\sigma} \rangle}(t^m)} e.$$
En particulier, $E=H(t, \sigma)  \otimes_{h^{\langle \widetilde{\sigma} \rangle}(t^m)} e$ est un corps et, par le \S\ref{ssec:prelim_2}, l'extension $E/H(t, \sigma)$ est galoisienne. De plus, puisque $e$ contient $\ell^{\langle \widetilde{\tau} \rangle} (t^m)$ et \eqref{eq:tensor_product} est vraie, on a $L(t, \tau) \subseteq E$.

Maintenant, le \S\ref{ssec:prelim_2} montre que l'application de restriction ${\rm{res}}^{E/H(t, \sigma)}_{e/h^{\langle \widetilde{\sigma} \rangle}(t^m)}$ est un isomorphisme. On peut alors consid\'erer l'isomorphisme compos\'e $\beta \circ {\rm{res}}^{E/H(t, \sigma)}_{e/h^{\langle \widetilde{\sigma} \rangle}(t^m)} : {\rm{Gal}}(E/H(t, \sigma)) \rightarrow G$. Le lemme \ref{lemma:link} fournit alors
$$\begin{array} {ccl}
\alpha_{\sigma, \tau} \circ \beta \circ {\rm{res}}^{E/H(t, \sigma)}_{e/h^{\langle \widetilde{\sigma} \rangle}(t^m)} & = & ({\rm{res}}^{L(t, \tau) / H(t, \sigma)}_{\ell^{\langle \widetilde{\tau} \rangle}/h^{\langle \widetilde{\sigma} \rangle}})^{-1} \circ \overline{\alpha}_{\sigma, \tau} \circ \beta \circ {\rm{res}}^{E/H(t, \sigma)}_{e/h^{\langle \widetilde{\sigma} \rangle}(t^m)} \\
& =& ({\rm{res}}^{L(t, \tau) / H(t, \sigma)}_{\ell^{\langle \widetilde{\tau} \rangle}/h^{\langle \widetilde{\sigma} \rangle}})^{-1} \circ {\rm{res}}^{e/h^{\langle \widetilde{\sigma} \rangle}(t^m)}_{\ell^{\langle \widetilde{\tau} \rangle}/ h^{\langle \widetilde{\sigma} \rangle}} \circ {\rm{res}}^{E/H(t, \sigma)}_{e/h^{\langle \widetilde{\sigma} \rangle}(t^m)}\\
& = & {\rm{res}}^{E/H(t, \sigma)}_{L(t, \tau)/H(t, \sigma)}, 
\end{array}$$
ce qui conclut la d\'emonstration du lemme.
\end{proof}

\begin{proof}[Preuve du th\'eor\`eme \ref{thm:3}]
Par 3), ${\rm{res}}^{\ell^{\langle \widetilde{\tau} \rangle} / h^{\langle \widetilde{\sigma} \rangle}}_{\ell_0/k_0}$ est un isomorphisme. Consid\'erons le probl\`eme de plongement fini
$$\alpha'={\rm{res}}^{\ell^{\langle \widetilde{\tau} \rangle} / h^{\langle \widetilde{\sigma} \rangle}}_{\ell_0/k_0} \circ \overline{\alpha}_{\sigma, \tau} : G \rightarrow {\rm{Gal}}(\ell_0/k_0)$$
sur $k_0$. Puisque $k_0$ est ample (par 3)) et $\alpha'$ est scind\'e (par 1)), on peut appliquer \cite[Main Theorem A]{Pop96} (voir aussi \cite[Theorem 1]{HJ98b}) pour obtenir que $\alpha'$ a une solution g\'eom\'etrique $\beta : {\rm{Gal}}(e/k_0(t^m)) \rightarrow G$ v\'erifiant $e \subseteq \ell_0((t^m))$, o\`u $m$ est l'ordre de $\tau$. Or, par 3) et puisque $e \cap \overline{k_0} = \ell_0$, l'application de restriction ${\rm{res}}^{eh^{\langle \widetilde{\sigma} \rangle}/h^{\langle \widetilde{\sigma} \rangle}(t^m)}_{e/k_0(t^m)}$ est un isomorphisme. Ainsi l'isomorphisme compos\'e
$$\beta \circ {\rm{res}}^{eh^{\langle \widetilde{\sigma} \rangle}/h^{\langle \widetilde{\sigma} \rangle}(t^m)}_{e/k_0(t^m)} : {\rm{Gal}}(eh^{\langle \widetilde{\sigma} \rangle}/h^{\langle \widetilde{\sigma} \rangle}(t^m)) \rightarrow G$$
est une solution g\'eom\'etrique \`a $\overline{\alpha}_{\sigma, \tau}$. De plus, comme $e \subseteq \ell_0((t^m))$, on a $eh^{\langle \widetilde{\sigma} \rangle} \subseteq \ell_0 h^{\langle \widetilde{\sigma} \rangle}((t^m))= \ell^{\langle \widetilde{\tau} \rangle}((t^m)).$ Il ne reste plus qu'\`a appliquer le lemme \ref{lemma:DL} pour conclure la d\'emonstration.
\end{proof}

\begin{rk} \label{rk:coro_1}
1) Pour $\tau={\rm{id}}_L$, le th\'eor\`eme \ref{thm:3} affirme que tout probl\`eme de plongement fini scind\'e sur un corps $H$ de dimension finie sur son centre $h$ admet une solution g\'eom\'etrique, si $h$ est ample, comme annonc\'e dans le th\'eor\`eme \ref{thm:main_1} de l'introduction.

\vspace{1mm}

\noindent
2) Si $H=L$ dans le th\'eor\`eme \ref{thm:3}, on a l'\'enonc\'e suivant, qui est un cas particulier de \cite[th\'eor\`eme A]{Beh21} et dont le cas $\sigma={\rm{id}}_H$ est \cite[th\'eor\`eme B]{DL20} :

\vspace{1mm}

\noindent
{\it{Soient $H$ un corps de dimension finie sur son centre $h$ et $\sigma$ un automorphisme de $H$ d'ordre fini. Supposons que $h^{\langle \widetilde{\sigma} \rangle}$ contienne un corps ample, o\`u $\widetilde{\sigma}$ est la restriction de $\sigma$ \`a $h$. Alors tout groupe fini est le groupe de Galois d'une extension galoisienne de $H(t, \sigma)$.}}

\vspace{1mm}

\noindent
3) Les 1) et 2) de cette remarque correspondent aux deux situations de la remarque \ref{rk:special_cases} dans lesquelles le corollaire \ref{prop:res_sigma} s'applique imm\'ediatement. Bien entendu, des corollaires du th\'eor\`eme \ref{thm:3} correspondant aux situations ``non triviales" de la remarque \ref{rk:special_cases} peuvent aussi \^etre donn\'es. Nous laissons ce travail au lecteur int\'eress\'e.
\end{rk}

\section{A propos de la r\'eduction faible$\rightarrow$scind\'e} \label{sec:fiber}

Dans cette derni\`ere partie, nous \'etendons la r\'eduction faible$\rightarrow$scind\'e au cas des corps de dimension finie sur leurs centres. Nous donnons aussi une variante du th\'eor\`eme \ref{thm:3} pour les probl\`emes de plongement finis admettant une solution faible (voir corollaire \ref{thm:4}).

\subsection{Extensions de  la r\'eduction faible$\rightarrow$scind\'e} \label{ssec:fiber_1}

Rappelons tout d'abord cette r\'eduction dans le cas commutatif (voir \cite[\S1 B) 2)]{Pop96} et \cite[\S2.1.2]{DD97b}) :

\begin{prop} \label{prop:fiber}
Soient $k$ un corps commutatif et $\alpha : G \rightarrow {\rm{Gal}}(\ell/k)$ un probl\`eme de plongement fini sur $k$. Pour toute solution faible ${\rm{Gal}}(\ell'/k) \rightarrow G$ \`a $\alpha$, il existe un probl\`eme de plongement fini scind\'e $\alpha' : G' \rightarrow {\rm{Gal}}(\ell'/k)$ sur $k$ v\'erifiant les propri\'et\'es suivantes :

\vspace{0.5mm}

\noindent
{\rm{1)}} ${\rm{ker}}(\alpha) \cong {\rm{ker}}(\alpha')$,

\vspace{0.5mm}

\noindent
{\rm{2)}} toute solution ${\rm{Gal}}(f'/k) \rightarrow G'$ \`a $\alpha'$ donne une solution ${\rm{Gal}}(f/k) \rightarrow G$ \`a $\alpha$ avec $f \subseteq f'$,

\vspace{0.5mm}

\noindent
{\rm{3)}} toute solution g\'eom\'etrique ${\rm{Gal}}(e'/k(t)) \rightarrow G'$ \`a $\alpha'$ fournit une solution g\'eom\'etrique $ {\rm{Gal}}(e/k(t)) \rightarrow G$ \`a $\alpha$ v\'erifiant $e \subseteq e'$.
\end{prop}

Nous d\'emontrons maintenant une premi\`ere g\'en\'eralisation :

\begin{prop} \label{prop:fiber1}
Soient $H$ un corps de dimension finie sur son centre, $\alpha : G \rightarrow {\rm{Gal}}(L/H)$ un probl\`eme de plongement fini sur $H$ et $\gamma : {\rm{Gal}}(L'/H) \rightarrow G$ une solution faible \`a $\alpha$. Il existe un probl\`eme de plongement fini scind\'e $\alpha' : G' \rightarrow {\rm{Gal}}(L'/H)$ sur $H$ v\'erifiant :

\vspace{0.5mm}

\noindent
{\rm{1)}} ${\rm{ker}}(\alpha) \cong {\rm{ker}}(\alpha')$,

\vspace{0.5mm}

\noindent
{\rm{2)}} toute solution ${\rm{Gal}}(F'/H) \rightarrow G'$ \`a $\alpha'$ donne une solution ${\rm{Gal}}(F/H) \rightarrow G$ \`a $\alpha$ avec $F \subseteq F'$.
\end{prop}

\begin{proof}[Preuve]
Consid\'erons le probl\`eme de plongement fini $\check{\alpha} : G \rightarrow {\rm{Gal}}(\ell/h)$ sur le centre $h$ de $H$ (voir \eqref{def_fep_down}), o\`u $\ell$ est le centre de $L$. Comme $\gamma$ est une solution faible \`a $\alpha$, $\check{\gamma}  : {\rm{Gal}}(\ell'/h) \rightarrow G$ (voir \eqref{def_sol_down}), o\`u $\ell'$ est le centre de $L'$, est une solution faible \`a $\check{\alpha}$ et $\mathcal{F}_H$ (voir \eqref{eq:red_norm}) n'a que le z\'ero trivial sur $\ell'$. La proposition \ref{prop:fiber} fournit alors un probl\`eme de plongement fini scind\'e $(\check{\alpha})' : G' \rightarrow$ ${\rm{Gal}}(\ell'/h)$ sur $h$ qui en v\'erifie les 1) et 2). Comme $\mathcal{F}_H$ n'a que le z\'ero trivial sur $\ell'$, on peut consid\'erer le probl\`eme de plongement fini ${\widehat{(\check{\alpha})'}} : G' \rightarrow {\rm{Gal}}(L'/H)$ sur $H$ (voir \eqref{def_fep_up}), que l'on note $\alpha'$. Notons que $\alpha'$ est scind\'e (car $(\check{\alpha})'$ l'est). De plus, puisque $(\check{\alpha})'$ v\'erifie le 1) de la proposition \ref{prop:fiber}, on a ${\rm{ker}}(\alpha')= {\rm{ker}}((\check{\alpha})') \cong {\rm{ker}}(\check{\alpha}) = {\rm{ker}}(\alpha)$, ce qui d\'emontre le 1). Maintenant, soit $\beta' : {\rm{Gal}}(F'/H) \rightarrow G'$ une solution \`a $\alpha'$. Alors $\check{\beta'} : {\rm{Gal}}(f'/h) \rightarrow G'$ est une solution \`a $\check{\alpha'} = (\check{\alpha})'$, o\`u $f'$ est le centre de $F'$, et $\mathcal{F}_ {H}$ n'a que le z\'ero trivial sur $f'$. Comme $(\check{\alpha})'$ v\'erifie le 2) de la proposition \ref{prop:fiber}, il existe une solution $\beta : {\rm{Gal}}(f/h) \rightarrow G$ \`a $\check{\alpha}$ telle que $f \subseteq f'$. Enfin, puisque $\mathcal{F}_{H}$ n'a que le z\'ero trivial sur $f$, $\hat{\beta} : {\rm{Gal}}((H \otimes_{h} f)/H) \rightarrow G$ (voir \eqref{def_sol_up}) est une solution \`a $\hat{\check{\alpha}}=\alpha$ telle que $H \otimes_{h} f \subseteq H \otimes_{h} f' = F'$. 
\end{proof}

Nous concluons cette partie avec la variante g\'eom\'etrique de la proposition \ref{prop:fiber1} suivante :

\begin{prop} \label{prop:fiber2}
Soient $H$ un corps de dimension finie sur son centre $h$ et $\sigma \in {\rm{Aut}}(H)$. Soient $\alpha : G \rightarrow {\rm{Gal}}(L/H)$ un probl\`eme de plongement fini sur $H$ et $\tau \in {\rm{Aut}}(L)$ \'etendant $\sigma$ et tel que
\begin{equation} \label{eq:produit_L}
\langle \widetilde{\tau}, {\rm{Gal}}(\ell/h) \rangle \cong \langle \widetilde{\tau} \rangle \times {\rm{Gal}}(\ell/h),
\end{equation}
o\`u $\widetilde{\tau}$ est la restriction de $\tau$ au centre $\ell$ de $L$. Soit $\gamma : {\rm{Gal}}(L'/H) \rightarrow G$ une solution faible \`a $\alpha$ et soit $\tau'$ un automorphisme de $L'$ d'ordre fini \'etendant $\tau$ et tel que
\begin{equation} \label{eq:produit_L'}
\langle \widetilde{\tau'}, {\rm{Gal}}(\ell'/h) \rangle \cong \langle \widetilde{\tau'} \rangle \times {\rm{Gal}}(\ell'/h),
\end{equation} 
o\`u $\widetilde{\tau'}$ est la restriction de $\tau'$ au centre $\ell'$ de $L'$. Alors il existe un probl\`eme de plongement fini scind\'e $\alpha' : G' \rightarrow {\rm{Gal}}(L'/H)$ sur $H$ v\'erifiant les deux conditions suivantes :

\vspace{0.5mm}

\noindent
{\rm{1)}} ${\rm{ker}}(\alpha) \cong {\rm{ker}}(\alpha')$,

\vspace{0.5mm}

\noindent
{\rm{2)}} toute solution $(\sigma, \tau')$-g\'eom\'etrique ${\rm{Gal}}(E'/H(t, \sigma)) \rightarrow G'$ \`a $\alpha'$ fournit une solution $(\sigma, \tau)$-g\'eom\'etrique ${\rm{Gal}}(E/H(t, \sigma)) \rightarrow G$ \`a $\alpha$ v\'erifiant $E \subseteq E'$.
\end{prop}

\begin{proof}[Preuve]
Comme \eqref{eq:produit_L} vaut, on peut consid\'erer le probl\`eme de plongement fini 
$$\alpha_{\sigma, \tau} : G \rightarrow {\rm{Gal}}(L(t, \tau) / H(t, \sigma))$$ 
sur $H(t, \sigma)$ (voir \eqref{eq:fep_geo}). De plus, comme \eqref{eq:produit_L'} vaut, $L'(t, \tau')/H(t, \sigma)$ est galoisienne \`a groupe de Galois fini et ${\rm{res}}^{L'(t, \tau')/H(t, \sigma)}_{L'/H}$ est un isomorphisme bien d\'efini (voir corollaire \ref{prop:res_sigma}). Ainsi
$$\gamma_{\sigma, \tau} = \gamma \circ {\rm{res}}^{L'(t, \tau')/H(t, \sigma)}_{L'/H} : {\rm{Gal}}(L'(t, \tau')/H(t, \sigma)) \rightarrow G$$
est une solution faible \`a $\alpha_{\sigma, \tau}$. Puisque $H(t, \sigma)$ est de dimension finie sur son centre (voir lemme \ref{lemma:easy}), la proposition \ref{prop:fiber1} s'applique et fournit un probl\`eme de plongement fini scind\'e $$\underline{\alpha}: G' \rightarrow {\rm{Gal}}(L'(t, \tau') / H(t, \sigma))$$ 
sur $H(t, \sigma)$ qui v\'erifie le 1) et le 2) de cette derni\`ere proposition. On peut alors consid\'erer le probl\`eme de plongement fini
$$\alpha' = {\rm{res}}^{L'(t, \tau')/H(t, \sigma)}_{L'/H} \circ \underline{\alpha} : G' \rightarrow {\rm{Gal}}(L'/H)$$
sur $H$. Comme $\underline{\alpha}$ est scind\'e, il en est de m\^eme pour $\alpha'$ et on a ${\rm{ker}}(\alpha') = {\rm{ker}}(\underline{\alpha}) \cong {\rm{ker}}(\alpha_{\sigma, \tau}) = {\rm{ker}}(\alpha)$. Maintenant, soit $\beta' : {\rm{Gal}}(E'/H(t, \sigma)) \rightarrow G'$ une solution $(\sigma, \tau')$-g\'eom\'etrique \`a $\alpha'$. Alors $\beta'$ est une solution \`a $\alpha'_{\sigma, \tau'} = \underline{\alpha}$. Comme $\underline{\alpha}$ v\'erifie le 2) de la proposition \ref{prop:fiber1}, il existe une solution $\beta : {\rm{Gal}}(E/H(t, \sigma)) \rightarrow G$ \`a $\alpha_{\sigma, \tau}$ v\'erifiant $E \subseteq E'$. Cela conclut la d\'emonstration puisque, par d\'efinition, $\beta$ est une solution $(\sigma, \tau)$-g\'eom\'etrique \`a $\alpha$.
\end{proof}

\subsection{Extension du th\'eor\`eme \ref{thm:3}} \label{ssec:fiber_2}

Notre dernier objectif est la variante du th\'eor\`eme \ref{thm:3} suivante, qui concerne les probl\`emes de plongement finis admettant une solution faible :

\begin{coro} \label{thm:4}
Soient $H$ un corps de dimension finie sur son centre $h$, $\sigma \in {\rm{Aut}}(H)$, $\alpha : G \rightarrow {\rm{Gal}}(L/H)$ un probl\`eme de plongement fini sur $H$ et $\tau \in {\rm{Aut}}(L)$ \'etendant $\sigma$ tels que

\vspace{0.5mm}

\noindent
{\rm{1)}} $\alpha$ admet une solution faible ${\rm{Gal}}(L'/H) \rightarrow G$ et il existe un automorphisme $\tau'$ de $L'$ d'ordre fini \'etendant $\tau$ et tel que $\langle \widetilde{\tau'}, {\rm{Gal}}(\ell'/h) \rangle \cong \langle \widetilde{\tau'} \rangle \times {\rm{Gal}}(\ell'/h)$, o\`u $\widetilde{\tau'}$ est la restriction de $\tau'$ au centre $\ell'$ de $L'$,

\vspace{0.5mm}

\noindent
{\rm{2)}} $\langle \widetilde{\tau}, {\rm{Gal}}(\ell/h) \rangle \cong \langle \widetilde{\tau} \rangle \times {\rm{Gal}}(\ell/h)$, o\`u $\widetilde{\tau}$ est la restriction de $\tau$ au centre $\ell$ de $L$,

\vspace{0.5mm}

\noindent
{\rm{3)}} $h^{\langle \widetilde{\sigma} \rangle}$ est un corps ample, o\`u $\widetilde{\sigma}$ est la restriction de $\sigma$ \`a $h$.

\vspace{0.5mm}

\noindent
Alors $\alpha$ a une solution $(\sigma, \tau)$-g\'eom\'etrique.
\end{coro}

\begin{proof}[Preuve]
Comme 1) et 2) sont vraies, la proposition \ref{prop:fiber2} s'applique et fournit un probl\`eme de plongement fini scind\'e $\alpha' : G' \rightarrow {\rm{Gal}}(L'/H)$ sur $H$ v\'erifiant le 1) et le 2) de cette derni\`ere proposition. Maintenant, comme 1) et 3) sont vraies et comme $\alpha'$ est scind\'e, $\alpha'$ a une solution $(\sigma, \tau')$-g\'eom\'etrique (voir th\'eor\`eme \ref{thm:3}). Il ne reste alors plus qu'\`a utiliser que le probl\`eme de plongement fini $\alpha'$ v\'erifie le 2) de la proposition \ref{prop:fiber2} pour achever la d\'emonstration.
\end{proof}

En particulier, si $H$ est un corps de dimension finie sur son centre $h$ et si $\alpha : G \rightarrow {\rm{Gal}}(L/H)$ est un probl\`eme de plongement fini sur $H$ admettant une solution faible, alors $\alpha$ a une solution g\'eom\'etrique, si $h$ est ample, comme annonc\'e \`a la fin de l'introduction.

\subsection{Remarque finale} \label{ssec:fiber_3}

Soit $\alpha : G \rightarrow {\rm{Gal}}(L/H)$ un probl\`eme de plongement fini sur un corps $H$ de dimension finie sur son centre $h$. Par le lemme \ref{lemma:DL}, $\alpha$ a une solution g\'eom\'etrique si $\check{\alpha}$ a une solution g\'eom\'etrique ${\rm{Gal}}(e/h(t)) \rightarrow G$ avec $e \subseteq \ell((t))$, o\`u $\ell$ est le centre de $L$.

La r\'eciproque est fausse en g\'en\'eral. En effet, supposons les conditions suivantes v\'erifi\'ees :

\noindent
{\rm{1)}} $h$ est ample, 

\noindent
{\rm{2)}} $\alpha$ poss\`ede une solution faible,

\noindent
{\rm{3)}} $\alpha$ n'est pas scind\'e.

\noindent
Comme 1) et 2) sont vraies, $\alpha$ a une solution g\'eom\'etrique (voir corollaire \ref{thm:4}). Cependant, supposons que $\check{\alpha}$ poss\`ede une solution g\'eom\'etrique $\beta : {\rm{Gal}}(e/h(t)) \rightarrow G$ v\'erifiant $e \subseteq \ell((t))$. Alors 0 n'est pas un point de branchement\footnote{Pour une extension galoisienne finie $e/k(t)$ de corps commutatifs, on dit que $t_0 \in \overline{k}$ est un {\it{point de branchement}} de $e/k(t)$ si $\langle t-t_0 \rangle$ est ramifi\'e dans la cl\^oture int\'egrale de $\overline{k}[t]$ dans $e\overline{k}$.} de $e/\ell(t)$. Puisque 0 n'est pas non plus un point de branchement de $\ell(t)/h(t)$, on en d\'eduit que 0 n'est pas un point de branchement de $e/h(t)$. De plus, le corps r\'esiduel de $e$ en n'importe quel id\'eal maximal $\mathfrak{P}$ contenant $t$ vaut $\ell$. Si $D_\mathfrak{P}$ d\'esigne le groupe de d\'ecomposition de $e/h(t)$ en l'id\'eal maximal $\mathfrak{P}$, on a un isomorphisme $\varphi_0 : D_\mathfrak{P} \rightarrow {\rm{Gal}}(\ell/h)$ d\'efini comme suit. Soit $B$ la cl\^oture int\'egrale de $h[t]$ dans $e$ et soit $\mathfrak{P}$ un id\'eal maximal de $B$ contenant $t$. La r\'eduction modulo $\mathfrak{P}$ de n'importe quel \'el\'ement $x$ de $B$ est not\'ee $\overline{x}$. On a donc $B/\mathfrak{P}= \ell$ et, pour $\sigma \in D_\mathfrak{P}$ et $x \in B$, on pose $\varphi_0(\sigma)(\overline{x}) = \overline{\sigma(x)}$. On v\'erifie alors facilement que l'on a $\check{\alpha} \circ \beta \circ \varphi_{0}^{-1}={\rm{id}}_{{\rm{Gal}}(\ell/h)}$, ce qui contredit 3).

Pour conclure, nous donnons un exemple de probl\`eme de plongement fini comme ci-dessus. Consid\'erons le groupe quaternionique $Q_8$, muni de la pr\'esentation $\langle i, j \, \mid \, i^4=1, i^2=j^2, jij^{-1} = i^{-1} \rangle$, et le probl\`eme de plongement fini $\alpha : Q_8 \rightarrow {\rm{Gal}}(\Qq((t))(\sqrt{2})/\Qq((t)))$ sur le corps ample $\Qq((t))$, d\'efini par $\alpha(i)(\sqrt{2}) =- \sqrt{2}$ et $\alpha(j)(\sqrt{2}) = \sqrt{2}$. Puisque $Q_8$ ne peut s'\'ecrire sous la forme $H_1 \rtimes H_2$, o\`u $H_1$ et $H_2$ sont des sous-groupes propres et non triviaux de $Q_8$, le probl\`eme de plongement fini $\alpha$ n'est pas scind\'e. Cependant, $\alpha$ a une solution faible. En effet, d'apr\`es \cite[Theorem 1.2.1]{Ser92}, $\Qq((t))(\sqrt{2+\sqrt{2}})/\Qq((t))$ est galoisienne de groupe $\Zz/4\Zz$ et on a $\Qq((t))(\sqrt{2}) \subseteq  \Qq((t))(\sqrt{2+\sqrt{2}})$. Si $\sigma$ est un g\'en\'erateur de ${\rm{Gal}}(\Qq((t))(\sqrt{2+\sqrt{2}})/\Qq((t)))$, on consid\`ere l'isomorphisme 
$$\beta : {\rm{Gal}}(\Qq((t))(\sqrt{2+\sqrt{2}})/\Qq((t))) \rightarrow \langle i \rangle$$ 
d\'efini par $\sigma \mapsto i$. Alors $\beta$ est une solution faible \`a $\alpha$. Enfin, puisque $\Qq((t))(\sqrt{2+\sqrt{2}})$ est de niveau infini, $\hat{\beta} : {\rm{Gal}}(H_{\Qq((t))(\sqrt{2+\sqrt{2}})}/H_{\Qq((t))}) \rightarrow Q_8$ est une solution faible au probl\`eme de plongement fini $\hat{\alpha} : Q_8 \rightarrow {\rm{Gal}}(H_{\Qq((t))(\sqrt{2})}/H_{\Qq((t))})$ sur le corps des quaternions $H_{\Qq((t))}$ \`a coefficients dans $\Qq((t))$ (voir remarque \ref{rk:bruno} pour la d\'efinition), mais $\hat{\alpha}$ n'est pas scind\'e.

\bibliography{Biblio2}
\bibliographystyle{alpha}

\end{document}